\documentclass[12pt,a4paper]{amsart}
\usepackage{amsfonts,amssymb,enumerate,bm,hhline,makecell,array,mathtools}
\usepackage{mathrsfs}
\usepackage{comment}
\usepackage[normalem]{ulem}
\usepackage[colorlinks=true,citecolor=blue!70!black, urlcolor=blue!70!black, pagebackref, linkcolor=blue!70!black]{hyperref}
\usepackage{tikz-cd}
\usepackage{amscd}
\usepackage[shortlabels]{enumitem}
\setlist[itemize]{noitemsep,nolistsep}
\usepackage{tensor}
\usepackage{tikz}
\usetikzlibrary{matrix,arrows}
\usepackage{extarrows}
\usepackage[toc,page]{appendix}
\usepackage[all,ps,cmtip,rotate]{xy} 
\usepackage{color}

\usepackage{enumitem}
\setlist[itemize]{noitemsep,nolistsep}
\setlist[enumerate]{noitemsep,nolistsep}
\usepackage{colortbl} 

\allowdisplaybreaks
\let\mathcal\mathscr

\def\Z{{\bf Z}}

\def\C{{\bf C}}
\def\A{{\bf A}}
\def\R{{\bf R}}
\def\Q{{\bf Q}}
\def\P{{\bf P}}

\def\HKm{Hyper-K\"ahler manifold}
\def\hk{hyper-K\"ahler}
\def\hKm{hyper-K\"ahler manifold}
\def\hkm{hyper-K\"ahler manifold}

\def\phi{\varphi}

\def\cA{\mathcal{A}}

\def\cD{\mathcal{D}}

\def\cL{\mathcal{L}}

\def\cO{\mathcal{O}}

\def\cR{\mathcal{R}}

\def\cW{\mathcal{W}}
\def\cX{\mathcal{X}}
\def\cY{\mathcal{Y}}
\def\cZ{\mathcal{Z}}

\def\gg{\mathfrak g}

\def\gS{\mathfrak S}

\def\lra{\longrightarrow}
\def\llra{\hbox to 10mm{\rightarrowfill}}
\def\lllra{\hbox to 15mm{\rightarrowfill}}

\def\llla{\hbox to 10mm{\leftarrowfill}}
\def\lllla{\hbox to 15mm{\leftarrowfill}}

\def\thra{\twoheadrightarrow}
\def\hra{\hookrightarrow}

\def\isom{\simeq}

\DeclareMathOperator{\isomlra}{\stackrel{{}_{\scriptstyle\sim}}{\lra}}

\DeclareMathOperator{\alg}{alg}

\DeclareMathOperator{\Aut}{Aut}

\DeclareMathOperator{\coker}{Coker}

\DeclareMathOperator{\Fix}{Fix}

\DeclareMathOperator{\Jac}{Jac}
\DeclareMathOperator{\Ext}{Ext}

\DeclareMathOperator{\Gr}{\mathsf{Gr}}

\DeclareMathOperator{\Hdg}{Hdg}
\DeclareMathOperator{\Hom}{Hom}
\DeclareMathOperator{\Id}{Id}
\DeclareMathOperator{\id}{Id}

\def\Re{\mathop{\rm Re}\nolimits}

\DeclareMathOperator{\NS}{NS}

\DeclareMathOperator{\td}{td}
\DeclareMathOperator{\pr}{\mathsf{pr}}

\def\res{\mathsf{res}}

\DeclareMathOperator{\Nef}{Nef}

\DeclareMathOperator{\Supp}{Supp}
\DeclareMathOperator{\Spec}{Spec}
\DeclareMathOperator{\Stab}{Stab}

\DeclareMathOperator{\SH}{SH}
\DeclareMathOperator{\Sing}{Sing}

\DeclareMathOperator{\Sym}{Sym}

\def\llra{\hbox to 10mm{\rightarrowfill}}
\def\lllra{\hbox to 15mm{\rightarrowfill}}

\def\bw#1#2{\textstyle{\bigwedge\hskip-0.9mm^{#1}}\hskip0.2mm{#2}}

\def\subset{\subseteq}

\newtheorem{lemm}{Lemma}[section]

\newtheorem{theo}[lemm]{Theorem}

\newtheorem{coro}[lemm]{Corollary}
\newtheorem{prop}[lemm]{Proposition}

\theoremstyle{definition}

\newtheorem{rema}[lemm]{Remark}

\theoremstyle{remark}
\newtheorem*{remark*}{Remark}
\newtheorem*{note*}{Note}

\newcommand{\sZ}{{\mathsf{Z}}}
\newcommand{\sW}{{\mathsf{W}}}

\def\Lkkk[#1]{{\Lambda_{\KKK^{[#1]}}}}
\def\kkk[#1]{{\KKK^{[#1]}}}
\DeclareMathOperator{\KKK}{{K3}}

\def\sss[#1]{{S^{[#1]}}}

\def\setminus{\smallsetminus}

\definecolor{orange}{rgb}{1,0.55,0}

\begin{document}
\title{On the fixed locus of the antisymplectic involution of an EPW cube}

\author[F.~Rizzo]{Francesca Rizzo}
\address{Universit\'e   Paris Cit\'e and Sorbonne Universit\'e, CNRS,   IMJ-PRG, F-75013 Paris, France}
 \email{{\tt francesca.rizzo@imj-prg.fr}}

 \date{\today}

 \frenchspacing

 \subjclass[2020]{14C20, 14D06, 14F45, 14J42, 14J50}
\keywords{Hyperk\"ahler manifolds, antisymplectic involutions, Lagrangian subvariety, moduli spaces, Bridgeland stability, LLV decomposition.}

\thanks{This project has received funding from the European
Research Council (ERC) under the European
Union's Horizon 2020 research and innovation
programme (ERC-2020-SyG-854361-HyperK)}

\begin{abstract}
EPW cubes are polarized \hk\ varieties of K$3^{[3]}$-type that carry an anti-symplectic involution. We study the geometry of the fixed locus $\sW_A$ of this involution and prove that it is a \emph{rigid} atomic Lagrangian submanifold. Our proof is based on a detailed description of certain singular degenerations of EPW cubes and the degeneration methods of Flapan--Macrì--O'Grady--Saccà.
\end{abstract}
 
\maketitle

 \section{Introduction}
 EPW cubes are \hk \ varieties of K$3^{[3]}$-type, constructed in \cite{IKKR19} by Iliev--Kapustka--Kapustka--Ranestad. They form a locally complete family of smooth projective \hk\ varieties of dimension 6 that carry a canonical polarization of square $4$ and divisibility $2$ and an antisymplectic involution.  The goal of this paper is twofold:
 \begin{itemize}
     \item describe a family of certain singular degenerations of EPW cubes;
     \item use this family and the degeneration methods of Flapan--Macrì--O'Grady--Saccà, who studied in \cite{FMOS22} and \cite{FMOS23} the fixed loci of antisymplectic involutions on \hk\ varieties with a polarization of square $2$ to compute some numerical invariants of the fixed locus.
 \end{itemize}

EPW cubes are defined as follows. Let $V_6$ be a 6-dimensional complex vector space  and consider the symplectic space $\bw3V_6$. For any Lagrangian subspace $A\subset\bw3V_6$, we consider the degeneracy loci
 $$
\sZ^{\ge i}_A\coloneqq \{[U]\in \Gr(3, V_6)\mid \dim(A\cap \bw2U\wedge V_6)\ge i\}
 $$
 in the Grassmannian $\Gr(3, V_6)$.

 When $A$  contains no decomposable vectors, that is, when the intersection $\P(A)\cap \Gr(3, V_6)\subset\P(\bw3V_6)$ is empty, the stratum $\sZ_A\coloneqq \sZ^{\ge 2}_A$ is a normal integral 6-fold  and its singular locus is the normal reduced 3-fold $\sZ^{\ge 3}_A$. Moreover, there exists a double cover
\begin{equation}\label{eqn:EPWcube}
     g_A \colon \widetilde \sZ_A \lra \sZ_A,
\end{equation}
  constructed in \cite[Theorem~5.7]{DK20}, branched along $\sZ^{\ge 3}_A$.
   The variety $\widetilde \sZ_A$ is called an EPW cube. 
  
  Assume that $A$ is general. Then, $\sZ^{\ge 4}_A$ is empty, $\sZ^{\ge 3}_A$ is smooth, and   $\widetilde \sZ_A$ 
is   a smooth \hk\ variety of K$3^{[3]}$-type (\cite{IKKR19}).
 The pullback $h$ to   $\widetilde \sZ_A$  of the class of a hyperplane section of $\Gr(3, V_6)$ is  an ample class of divisibility $2$ and square $4$    with respect to the Beauville--Bogomolov--Fujiki quadratic form on $H^2(\widetilde \sZ_A, \Z)$. Moreover, the
  double cover $g_A$ defines a regular antisymplectic involution $\iota_A$ on $\widetilde \sZ_A$ whose fixed locus is  $\sW_A\coloneqq g_A^{-1}(\sZ^{\ge 3}_A)$. It is a smooth 
  3-dimensional Lagrangian submanifold of $\widetilde \sZ_A$. In this paper, we compute some cohomological  invariants of  $\sW_A$.

In the first part of the paper (Sections~\ref{sec:hkK3} and~\ref{sec:EPW}), we study the cohomology ring $H^\bullet(\widetilde \sZ_A, \Q)$ and the action of $\iota^*_A$ on it   
to prove the following theorem.

\begin{theo}\label{thm:Xtop}
    Let $\widetilde \sZ_A$ be a smooth EPW cube with associated involution $\iota_A$. The fixed locus~$\sW_A$ of~$\iota_A$ is a smooth threefold of general type, with canonical bundle $\omega_{\sW_A}=\cO_{\sW_A}(2)$ and self-intersection number in $\widetilde \sZ_A $ given by
    $$
        [\sW_A]^2 = -\chi_\mathrm{top}(\sW_A) = 1200.
    $$
\end{theo}

Our method of proof uses the action of the so-called Looijenga--Lunts--Verbitsky (LLV) algebra~$\gg(X)$ on the  
 cohomology ring $H^\bullet  (X, \Q)$ of a \hk \ variety $X$. This action induces a decomposition, called the LLV decomposition, of $H^\bullet  (X, \Q)$ into irreducible representations of $\gg(X)$.
 One of these irreducible representations is the so-called Verbitsky component $SH(X, \Q)$, defined as the subalgebra of $H^\bullet  (X, \Q)$ generated by $H^2(X, \Q)$. 

Even if the cohomology ring of $X$ is determined by $H^2(X, \Q)$ and the LLV decomposition,   it is not clear in general whether the action in cohomology of an automorphism of $X$ is induced by its action on $H^2(X, \Q)$.
{Following \cite{BGGG25}, we show that for any finite order automorphism on a \hk\ sixfold of K$3^{[3]}$-type, the action in cohomology is the \emph{natural} one induced by the action on $H^2(X, \Q)$, up to the choice of a sign. For EPW cubes, the computation of the topological Euler characteristic of the fixed locus of Theorem~\ref{thm:Xtop} shows that the action of the antisymplectic involution in cohomology is indeed natural (see Lemma~\ref{lemm:ClassOfWA}).}

 For this purpose, we study the LLV decomposition and the Hodge classes on very general \hk \ varieties of K$3^{[3]}$-type. In particular, in the polarized case, we determine the relations between the {Chern} classes   (Proposition~\ref{[prop:relations}). 
We deduce a formula for the projection $\overline{[W]}$ of the cohomology class $[W]\in H^6(X, \Q)$ to $\SH(X,\Q)$, when $W$ is a Lagrangian submanifold of a very general polarized \hk \ variety $X$ (see Proposition~\ref{prop:ClassLagrangian}). 

As observed in Remark~\ref{rem:atomic}, the Lagrangian submanifold $\sW_A$ of $\widetilde \sZ_A$ is atomic in the sense of \cite[Definition~1.1]{Bec25}. After \cite[Section~7.4]{Mar24}, a natural question is to ask whether  it is rigid inside $\widetilde \sZ_A$. In the second part of the paper (Sections~\ref{sec:deg} and~\ref{sec:DegWA}),  we use the degeneration method of \cite{FMOS22} and \cite{FMOS23} to show that the answer is positive.

\begin{theo}
    Let $\widetilde \sZ_A$ be a smooth EPW cube with associated involution $\iota_A$. The fixed locus $\sW_A$ of~$\iota_A$ is an irreducible rigid atomic Lagrangian submanifold.
\end{theo}

The degeneration we use is interesting in its own right and complements the study in \cite{Riz25} of another family of singular EPW cubes. It is constructed as follows. Let $(S,L)$ be a polarized K3 surface of degree $2$, with associated double cover $S\to \P^2$ branched over a smooth sextic curve $\Gamma\subset \P^2$. The class $2L-\delta$ is big and nef on the Hilbert scheme $S^{[3]}$, and it has square $4$ and divisibility $2$. Therefore, it induces a contraction $S^{[3]}\to X$, and the involution of $S$  given by the double cover induces an involution $\iota$ on $X$.

The singular variety $X$ is a degeneration of EPW cubes. More precisely, we construct a family $\cX\to B$,  whose general fibers are (smooth) EPW cubes and whose central fiber is $X$,  which carries an involution $\iota_\cX$ that restricts to $\iota_A$ on each smooth fiber $\widetilde \sZ_A$, and to $\iota$ on the special fiber.

Results from \cite{BM14a} on contractions of moduli spaces of sheaves on K3 surfaces   allow  us to explicitly describe $X$, the involution $\iota$, and its fixed locus. The latter is a disjoint union of two components, one of dimension $2$ (which does not deform), and $F_3$, of dimension $3$. The normalization of the component $F_3$ is a birational morphism
$$
\Gamma^{(3)}\lra F_3 
$$
which restricts to an étale double cover $E_3\to \Sing(F_3)$ on a divisor $E_3$ of $\Gamma^{(3)}$.

One component of the fixed locus of $\iota_{\cX}$ induces a family $\cY\to B$ of relative dimension~$3$, with general fiber $\sW_A$ and special fiber $F_3$. This family   is not semistable, so the Clemens--Schmid method cannot be used to determine the Hodge decomposition of $\sW_A$. However, the fibers of $\cY\to B$ are semi-log canonical, hence the 
   numbers $h^i(\cY_b, \cO_{\cY_b})$ are constant. In particular, the ``external'' Hodge numbers of $\sW_A$ are given by the dimension of the cohomology groups of the structure sheaf of $F_3$. 

As in \cite{FMOS23}, the most delicate part is showing that the (schematic) central fiber of $\cY\to B$ is equal to $F_3$, and in particular proving that it is reduced. 
This is done in Section~\ref{sec:localDescr} where, using the Kuranishi map, we determine the local structure the fixed locus of $\iota$.
\subsection*{Structure of the paper}
The paper is organized as follows. 

In Section~\ref{sec:hkK3}, we recall the LLV structure for the cohomology ring $H^*(X, \Z)$ of a very general polarized \hk \ variety $X$ of K$3^{[3]}$-type. In particular, we study Hodge classes  and the action of automorphisms (of finite order) on $H^*(X, \Z)$.

In Section~\ref{sec:EPW}, we specialize the previous analysis to the case of EPW cubes, to obtain Theorem~\ref{thm:Xtop}. We also compute the cohomology class   $[\sW_A]$    (Proposition~\ref{lemm:ClassOfWA}) and the Chern classes of $\sW_A$ (Section~\ref{sec:ChernNumbers}).

In Section~\ref{sec:deg}, we construct the degeneration $\cX\to D$ and describe the special fiber $X$ and its involution.

Finally, in Section~\ref{sec:DegWA}, we construct and study the degeneration $\cY\to D$, and prove that $\sW_A$ is irreducible and that its first cohomology group $H^1(\sW_A, \Q)$ is trivial.

\subsection*{Acknowledgements} 
I would like to thank my advisor, Olivier Debarre, for suggesting me this problem and for his insights and hints. I am also grateful to Emanuele Macrì for the helpful discussions on the study of fixed loci in degenerations, to Jieao Song for discussing with me Remark~\ref{rem:Jieao} and to Mirko Mauri for pointing out Kollár’s result \cite[Corollary~2.64]{Kou93}.
I thank Pietro Beri for useful discussions, and Francesco Denisi, Giovanni Mongardi, and Vanja Zuliani for their interest in this work.

\section{\HKm s of $\KKK^{[3]}$-type}\label{sec:hkK3}
Let $X$ be a \hKm\ of $\KKK^{[3]}$-type. Its 
nonzero Hodge numbers are
$$
 \xymatrix
 @R=7pt
 @C=1pt
  @M=3pt{
b_0=1&& 
& & & & & & & 1  \\     
b_2=23&&
&&&&& 1
&& 21
&&1  \\    
b_4=299&&
&&&1
&&22
&& 253
&&22
&&1   \\   
b_6=2554&\qquad&
&1
&&21
&&253
&&2004
&&253
&&21
&&1
}    
 $$
We denote by $\SH^{2\bullet}(X,\Q)$ the image of the canonical map
$$\Sym^\bullet \! H^{2}(X,\Q)\lra H^{2\bullet}(X,\Q)$$
(the so-called Verbitsky component). Let $c_2,c_4,c_6$ be the nonzero Chern classes of $X$. One   has (\cite[Lemma~1.5]{Mar20})
$$c_2 \in \Sym^2\! H^2(X,\Q)=\SH^4(X,\Q)\subset H^4(X,\Q).$$
 The nonzero Chern numbers are (\cite[Theorem~0.1]{EGL01})
\begin{equation}\label{chernnumbers}
    c_6=3200\ ,\ c_2c_4=14720\ ,\ c_2^3=36800.
\end{equation}

 Recall that, given a class $\alpha\in H^{4k}(X, \Q)$ that remains of type $(2k,2k)$ on all small deformations of $X$, the generalized Fujiki constant is a nonzero rational number $C(\alpha) $ that satisfies
 \begin{equation}\label{eqn:defgenFujConst}
 \forall \beta\in H^2(X, \Q) \qquad \int_X \alpha\cdot \beta^{6-2k} = C(\alpha)q_X(\beta)^{3-k}.
\end{equation}
From \cite[Lemma~2.1.7]{Son22a}, we have\footnote{{All Chern numbers and Fujiki constants for \hk\ manifolds of $\KKK^{[n]}$-type can be computed using https:/\!/github.com/8d1h/bott.}}
\begin{equation}\label{eqn:genFujConst}
    C(1) = 15, \quad C(c_2) = 108, \quad
   C(c_2^2) = 1200,  \quad C( c_4) = 480 .
\end{equation}

    Polarizing both terms of Equation~\eqref{eqn:defgenFujConst}, we obtain that   for all $\beta_1, \dots, \beta_{6-2k} \in H^2(X, \Q)$,
 \begin{equation}\label{eqn:PolFujiki}
  \int_X\alpha\beta_1\cdots\beta_{6-2k} = \frac{C(\alpha)}{(6-2k)!}\sum_{\sigma\in \gS_{6-2k}} q(\beta_{\sigma(1)}, \beta_{\sigma(2)})\cdots q(\beta_{\sigma(6-2k-1)}, \beta_{\sigma(6-2k)}),
 \end{equation}
 where $\gS_{6-2k}$ is the group of permutations of $6-2k$ elements.

\subsection{Hodge classes on very general \hk\ manifolds of K$3^{[3]}$-type}\label{sec:hodgeVG}
 Following \cite[Section 2]{GKLR22}, we denote by $\bar V$ the vector space $H^2(X, \Q)$ with the (rational) Beauville--Bogomolov--Fujiki form $q$, and by $V$ the Mukai completion 
 \begin{equation}\label{eqn:MukaiComplet}
     V = H^2(X, \Q) \oplus U
 \end{equation} with the quadratic form $q\oplus\begin{pmatrix} 0 & 1 \\1 &0\end{pmatrix}$, so that $U\simeq\Q^{\oplus 2}$ with the structure of a hyperbolic plane.

We follow \cite[Example~14]{Mar02} and  \cite[Remark~3.3]{GKLR22}. The so-called LLV-algebra $\mathfrak{g} = \mathfrak{so}(V)$ acts on $H^\bullet(X,\Q)$  and its
  decomposition into irreducible representations (which are also sub-Hodge structures) is\footnote{See also https:/\!/webusers.imj-prg.fr/$\sim$jieao.song/static/llv.html}
\begin{equation}\label{eqn:irred_rap}
    H^\bullet(X, \Q) = V_{(3)} \oplus V_{(1,1)},
\end{equation}
where $V_{(3)}$ is the Verbitsky component $\SH^{2\bullet}(X, \Q)$ and $V_{(1,1)}$ is $\bw2 V$. Furthermore, under the action of $\mathfrak{so}(\bar V)$,  the cohomology groups decompose further into the following  representations     
$$
\begin{array}{rcl}
     H^0(X,\Q)&=&  \Q,\\
     H^2(X,\Q)&=&  \bar V,\\
     H^4(X,\Q)&=&\Sym^2 \!\bar V\oplus \bar V ,\\
     H^6(X,\Q)&=&\Sym^3 \!\bar V\oplus \bw2{\bar V}\oplus \Q\eta,\\
     H^8(X,\Q)&=&\Sym^2 \!\bar V\oplus \bar V ,\\
     H^{10}(X,\Q)&=&  \bar V,\\
     H^{12}(X,\Q)&=&  \Q,
\end{array}
$$  
where the summands in the fourth line are pairwise orthogonal with respect to the intersection pairing.\footnote{The class $\eta$  was studied in \cite[Proposition~6.1]{HHT12}.} There are further decompositions  
\begin{align*}
    \Sym^2\!\bar V &= \bar V_{(2)}\oplus \Q,\\
    \Sym^3\!\bar V &= \bar V_{(3)} \oplus V,
\end{align*}
into irreducible representations, whereas $\bw2{\bar V} = \bar V_{(1,1)}$ is already irreducible.

For $X$ very general (so nonprojective), the (special) Mumford--Tate algebra of $X$ is equal to the   algebra $\mathfrak{so}(\bar V)$ (\cite[Proposition~2.38]{GKLR22}). Therefore the space of Hodge classes in $H^\bullet(X, \Q)$ coincides with its $\mathfrak{so}(\bar V)$-invariant part (\cite[Section~2.3]{GKLR22}): it is given by
\begin{equation}\label{eta}
    \begin{array}{rcl}
     \Hdg^0(X,\Q)&=&  \Q,\\
     \Hdg^2(X,\Q)&=&  0,\\
     \Hdg^4(X,\Q)&=& \Q c_2,\\
     \Hdg^6(X,\Q)&=&\Q\eta,\\
     \Hdg^8(X,\Q)&=&\Q c_4 ,\\
     \Hdg^{10}(X,\Q)&=&0 ,\\
     \Hdg^{12}(X,\Q)&=&  \Q.
\end{array}
\end{equation}
In particular, the classes $c_4$ and $c_2^2$ are proportional and the proportionality coefficient    can be determined using the constants given in~\eqref{eqn:defgenFujConst}: one finds
\begin{equation}\label{c4c2}
    c_2^2=\frac{C(c_2^2)}{C(c_4)}\,c_4=  \frac{1200}{480}\, c_4= \frac{5}{2} c_4.
\end{equation}
 This remains true on any \hk\ manifold  of K$3^{[3]}$-type.

 \subsection{Action of automorphisms on the LLV algebra}\label{sec:AutOnLLV}

Given an automorphism $\varphi$ of finite order of $X$, we want to study the induced automorphism   $\varphi^*$ of $H^\bullet(X, \Q)$. 

The automorphism $\varphi^*_2$ of $\bar V = H^2(X, \Q)$ extends to an automorphism $\gamma$  of $V$ by acting as the identity on the hyperbolic plane $U$ (see Equation~\eqref{eqn:MukaiComplet}). This induces automorphisms $\gamma_\lambda$ of any $\mathfrak{g}$-representation $V_\lambda$. We compare in Lemma~\ref{lemm:NaturalOpppositeAction} the automorphism $\varphi^*$ with the automorphism $\gamma_{(3)}\oplus \gamma_{(1,1)}$ of $H^\bullet(X, \Q)$ obtained from the isomorphism \eqref{eqn:irred_rap}. 
{In Lemma~\ref{lemm:ClassOfWA}, we   show that the two actions are equal for the antisymplectic involution associated to EPW cubes.}

\begin{lemm} Let $X$ be a  \hk\ manifold and let $\varphi$ be an  automorphism of $X$.
    For each $\mathfrak{g}$-stable factor $W$ of $H^ \bullet(X, \Q)$, $\varphi^*(W)$ is again $\mathfrak{g}$-stable and isomorphic to $W$.
\end{lemm}

\begin{proof}
   By \cite[Lemma~2.3]{Flo22}, the application $h\mapsto (\varphi^*)h(\varphi^*)^{-1}$ induces an automorphism of $\mathfrak{g}$. In particular, $\mathfrak{g}\varphi^*(W) = \varphi^*\mathfrak{g}(W) = \varphi^*(W)$ for any $\mathfrak{g}$-invariant subspace $W$.
\end{proof}

The lemma implies that, in our case, the automorphism $\varphi^*$ of $H^\bullet(X, \Q)$ preserves the irreducible components $V_{(3)}$ and $V_{(1,1)}$, hence it decomposes as $\varphi^*=\varphi^*_{(3)} \oplus \varphi^*_{(1,1)}$.

Since the map from the Verbitsky component to  $H^*(X, \Q)$ is given by cup product, one has $\varphi^*_{(3)}=\gamma_{(3)}$.
For $\varphi^*_{(1,1)}$, the same argument as in \cite[Proposition~3.3]{BGGG25} shows the following.

\begin{lemm}\label{lemm:NaturalOpppositeAction}
    Let $\varphi\in \Aut(X)$ be an automorphism of order $n$ and let $\varphi^*_{(1,1)}$ be the induced automorphism of $V_{(1,1)}$.
    If $n$ is odd, $\varphi^*_{(1,1)}=\gamma_{(1,1)}$, while if $n$ is even, $\varphi^*_{(1,1)}=\pm \gamma_{(1,1)}$.
\end{lemm}

\begin{proof}
We follow the proof of \cite[Proposition~3.3]{BGGG25}, borrowing its notation. 
 Since $\varphi^*_{(1,1)}$ with $\gamma_{(1,1)}$ are both   automorphisms of $V_{(1,1)}$, they satisfy
    $$
        \varphi^*_{(1,1)}(L_\alpha\cdot x) = L_{\varphi^*(\alpha)}\cdot \varphi^*_{(1,1)}(x), \quad\text{and}\quad \gamma_{(1,1)}(L_\alpha\cdot x) = L_{\varphi^*(\alpha)}\cdot \gamma_{(1,1)}(x) 
    $$
    for all $\alpha\in H^2(X, \Z)$ that define a $\mathfrak{sl}_2$-triple. In particular,
    $$\varphi^*_{(1,1)}\circ \gamma_{(1,1)}^{n-1}(L_\alpha\cdot x) = \varphi^*_{(1,1)}(L_{(\varphi^*)^{n-1}(\alpha)}\cdot \gamma_{(1,1)}^{n-1}(x)) = L_\alpha \cdot \varphi^*_{(1,1)}\circ\gamma_{(1,1)}^{n-1}(x).$$
    
   Therefore $\varphi^*_{(1,1)}\circ \gamma_{(1,1)}^{n-1}$ is a $\gg$-isomorphism of $V_{(1,1)}$, hence, it is equal to $\lambda\cdot \id$ by Schur’s lemma, for some $\lambda\in \Q$. Comparing orders, we obtain   $\lambda^n = \id$, hence $\lambda$ is $1$ for $n$ odd, and equal to $\pm 1$ for even $n$.
\end{proof}

\subsection{Hodge classes on very general polarized \hk\ manifolds of K$3^{[3]}$-type}\label{sec:hodgepol}

We now consider a very general \emph{polarized} \hk\ manifold $(X,h)$ of K$3^{[3]}$-type. In this case, the vector space $\bar V$ decomposes as $\bar T\oplus \Q h$, where $h$ is the polarization and $\bar T$ is the transcendental Hodge structure. The cohomology groups decompose  further as 
$$
\begin{array}{rcl}
     H^0(X,\Q)&=&  \Q,\\
     H^2(X,\Q)&=&  \bar T \oplus \Q,\\
     H^4(X,\Q)&=&(\Sym^2 \!\bar T \oplus \bar T \oplus \Q) \oplus ( \bar T \oplus \Q) ,\\
     H^6(X,\Q)&=&(\Sym^3 \!\bar T\oplus \Sym^2\!\bar T \oplus  \bar T \oplus \Q) \oplus (\bw2{\bar T} \oplus  \bar T) \oplus \Q.\\
\end{array}
$$  

In this case, the (special) Mumford--Tate algebra of $X$ is equal to $\mathfrak{so}(\bar  T)$, and the $\mathfrak{so}(\bar  T)$-invariant part of $H^\bullet(X, \Q)$ is the space of Hodge classes. 

One has  $\Hdg^2(X,\Q)=\Q h$.  

The space  $\Hdg^4(X,\Q)$  has dimension 3 (recall that the invariant part of $\Sym^2\bar T$ has dimension 1). We denote by $\lambda$ a generator of the trivial factor outside the Verbitsky component. 

The space  $\Hdg^6(X,\Q)$ also has dimension 3 and intersects the Verbitsky component in dimension $2$. Up to multiplication by a nonzero rational number, we obtain that $h\lambda$ is equal to the Hodge class $\eta$ from~\eqref{eta}. By \cite[Proposition 6.1]{HHT12}, its   square is $4$ (modulo squares of nonzero rational numbers).

The following lemma gives us a basis of $\mathrm{Hdg}^\bullet(X, \Q)$.

\begin{lemm}\label{lemm:degree6}
The families $(h^2, c_2)$ in  $\Hdg^4(X, \Q)$ and $(h^3, hc_2)$ in $ \Hdg^6(X, \Q)$  are linearly independent.
\end{lemm}

\begin{proof}
The statement in degree $4$ follows from the one in degree $6$, since any nontrivial linear relation between $h^2$ and $c_2$ would give a nontrivial linear relation between $h^3$ and $hc_2$.

We show that $h^3$ and $hc_2$ are linearly independent. 
  From Equation~\eqref{eqn:PolFujiki}, we get that there exist nonzero constants  $x$, $y$, and $z$
such that \begin{align*}
       \int_X h^3\beta^3 &= xq(h) q(\beta) q(h, \beta) + yq(h,\beta)^3, \\ 
       \int_X hc_2 \beta^3 &= z q(h,\beta)q(\beta)
 \end{align*}
  for all $\beta\in H^2(X,\Z)$.
 Assume that the classes  $h^3$ and $hc_2$ are proportional. Then for any $\beta$ such that $q(\beta)=0$, we have $q(h,\beta)=0$, which is absurd since $q$ is nondegenerate, hence the quadric $q(\beta)=0$ is not contained in any hyperplane.
\end{proof}

Thus, the spaces of Hodge classes on a very general polarized \hk\ sixfold $(X,h)$ of K$3^{[3]}$-type is \begin{align}\label{eqn:hodge}
        \Hdg^0(X,\Q) &= \Q ,\nonumber \\ 
    \Hdg^2(X,\Q) &= \Q h,\nonumber \\ 
      \Hdg^4(X,\Q) &=\Q h^2\oplus \Q c_2\oplus \Q \lambda, \nonumber\\
      \Hdg^6(X,\Q) &=\Q h^3\oplus \Q hc_2\oplus \Q h\lambda,\\
      \Hdg^8(X,\Q) &=\Q h^4\oplus \Q h^2c_2\oplus \Q h^2\lambda,\nonumber\\
      \Hdg^{10}(X,\Q) &=\Q h^5,\nonumber\\
            \Hdg^{12}(X,\Q) &=\Q h^6,\nonumber
 \end{align}
 where the class $\lambda$ depends on the polarization $h$, while the class $h\lambda$ is equal to the class $\eta$, and is orthogonal to the Verbitsky component $\SH^3(X, \Q)$. \\

The intersection $\Hdg^\bullet(X, \Q)\cap \SH^{2\bullet}(X, \Q)$ is the $\Q$-subalgebra of $H^\bullet(X, \Q)$ generated by $\{h,c_2,c_4,c_6\}$.
We determine the relations in each degree.

\begin{prop}\label{[prop:relations}
\textnormal{(a)} There are no relations in degree $\le 6$.

\noindent\textnormal{(b)} In degree $8$, the relations are
 $$
        c_4 = -\frac{160}{q(h)^2}h^4 + \frac{80}{3q(h)}h^2c_2, \quad c_2^2=\frac52 c_4.
    $$ 
    
\noindent\textnormal{(c)} In degree $10$, the relations are
  $$
        h^3c_2= \frac{36}{5q(h)}\,h^5, \quad h c_2^2= \frac{80}{q(h)^2}\,h^5, \quad h c_4=\frac{32}{q(h)^2}\, h^5.
    $$

    \noindent\textnormal{(d)} In degree $12$, we have
$$ 
      h^6 = 15\, q(h)^3 , \quad
      h^4c_2 =108 \,q(h)^2 , \quad
      h^2c_2^2 =1200\, q(h)  ,\quad
      h^2c_4 =480\, q(h).
$$ 
\end{prop}

\begin{proof}
The relations in degree $12$ follow from the values of the generalized Fujiki constants given in \eqref{eqn:genFujConst} and their definition~\eqref{eqn:defgenFujConst}. The relations in degree $10$ then follow from the fact that the vector space $\Hdg^{10}(X, \Q)$ has dimension one. For example, we obtain
$$h^3c_2=\frac{h^4c_2}{h^6}\, h^5=\frac{108 \,q(h)^2}{15\, q(h)^3}\, h^5= \frac{36}{5q(h)}\,h^5
$$
and similarly for the other relations.

In   degree $8$, the second relation is~\eqref{c4c2}. For the first relation, write  $c_4 = xh^4 + yh^2c_2$ for some $x,y\in\Q$. Taking the product with $h^2$, we get
$$480\, q(h)=15\, xq(h)^3+108 \,yq(h)^2.$$
Taking the product with $c_2$ and using the Chern numbers given in~\eqref{chernnumbers}, we get
$$14720= 108 \,xq(h)^2+1200\, yq(h).$$
Solving this system of equations for $x$ and $y$, we obtain the desired relation.
\end{proof}
  \begin{rema}
     One also have the following relations involving $\lambda$. The class $\eta = h\lambda$ is orthogonal to the Verbitsky component $\SH^3(X, \Q)$ and (up to multiply $\eta$ by a rational number), we can suppose $\eta^2 = 4$. Namely, in degree $12$ we have the following relations
     $$
        h^2\lambda^2 = 4, \qquad h^4\lambda = 0, \qquad h^2c_2\lambda = 0, \qquad c_2^2\lambda = 0,
     $$
     where the last equality is obtain from the fact that $c_2^2$ is a linear combination of $h^4$ and $h^2c_2$.
     
     Since $\Hdg^5(X) = \Q\cdot h^5$, the above equalities together with Proposition~\ref{[prop:relations} imply
     $$
        h\lambda^2 = \frac{4}{15 q(h)^3} h^5, \qquad h^3\lambda= 0, \qquad hc_2\lambda = 0.
     $$
 \end{rema}

 We conclude this section by studying the class of Lagrangian submanifolds of \hk\ manifolds of $\kkk[3]$-type.

\begin{prop}\label{prop:ClassLagrangian}
    Let $W $ be a Lagrangian submanifold of a very general polarized \hkm\ $(X,h)$ of $\kkk[3]$-type. The orthogonal  projection to the Verbitsky component $ \mathrm{SH}^{6}(X, \Q)$ of the class $[W]\in H^{6}(X, \Q) $   is equal to  
    $$
        \overline{[W]} \coloneqq \frac{[W]\cdot h^3}{72\cdot q(h)^2}\left(\frac{12}{q(h)}h^3-hc_2\right).
    $$
\end{prop}

\begin{proof}
The class $[W]$ is a 3-dimensional Hodge class  hence, by Equation~\eqref{eqn:hodge}, it can be written as
$$
[W] = ah^3+bhc_2+ch\lambda,
$$
where $a,b,c\in \Q$, and $h\lambda$ 
is orthogonal to the Verbitsky component. In particular, $\overline{[W]} = ah^3+bhc_2$ and, for any class $\gamma\in \SH^6(X, \Q)$, we have $\gamma\cdot[W]=\gamma\cdot \overline{[W]}$.

 We apply this to the class $\gamma = h\sigma\bar\sigma$ in $\SH^6(X, \Q)$, where $\sigma$ is a symplectic form on $X$. Since $W$ is Lagrangian, it satisfies $\gamma\cdot [W]=0$ and thus $\gamma\cdot \overline{[W]}=0$. 
Using Equation~\eqref{eqn:PolFujiki}, the Fujiki constants computed in Equation~\eqref{eqn:genFujConst}, and the fact that $q(\sigma, h)=0$, we compute
\begin{align*}
    \gamma\cdot \overline{[W]}&= a\int_X h^4\sigma\bar\sigma + b\int_X c_2h^2\sigma\bar\sigma\\
    &=a\frac{C(1)}{6!}6\cdot4!q(h)^2q(\sigma,\bar \sigma) + b \frac{C(c_2)}{4!}8q(h)q(\sigma, \bar \sigma)\\
    &=q(h)q(\sigma, \bar\sigma)(3q(h)a + 36b).
\end{align*}
Since $\gamma\cdot \overline{[W]}=0$ and $q(h)q(\sigma,\bar\sigma)>0$, we obtain $a = -\frac{12}{q(h)}b$.

Finally, we determine $b$ in terms of the degree   of $W$. Indeed, we have
$$
  [W]\cdot h^3 =  \overline{[W]}\cdot h^3 = b\Bigl( hc_2 -\frac{12}{q(h)}h^3\Bigr)\cdot h^3 = -72\,q(h)^2b,
$$
where, in the last equality, we use the numbers computed in Proposition~\ref{[prop:relations}. This gives the value of $b$ and proves the proposition.
\end{proof}

 \begin{rema}\label{rem:Jieao}
 Jieao Song pointed out to me that the above formula generalizes to any polarized \hk\ variety $X$ of dimension $2n$ with polarization $h$, and any Lagrangian subvariety $F$ {deforming with the polarization}. In particular, following the same strategy of \cite[Theorem~2.2.4]{Son22a}, he shows that the orthogonal projection to the Verbitsky component $\SH^{2n}(X, \Q)$ of the class $[F]\in H^{2n}(X, \Q)$ is equal to
 \begin{equation}\label{formJieao}
     \overline{[F]} = \frac{[F]\cdot h^n}{q(h)^n} \cdot \frac{\mu^n}{c_X}\overline{[\exp(h/\mu)\td_X^{1/2}]_n}, \quad \text{where $\mu \coloneqq \sqrt{\frac{-q(h)}{2r_X}}$.}
 \end{equation}
      Here $c_X = \frac{C(1)}{(2n-1)!!}$ and  $r_X \coloneqq \frac{C(c_2)2^nn!(2n-1)}{(2n)!24C(1)}$ is the constant introduced in \cite[Equation~(3.1)]{Bec23}.
      
The key point is to generalize \cite[Lemma~2.2.1]{Son22a} by showing that for any class $\alpha\in H^2(X, \Q)$ there is an equality $$(\iota^*\alpha)^{n}\cdot [F] = \frac{[F]\cdot h^n}{q(h)^n} q(\alpha, L)^n.$$

By writing $\td_X^{1/2} = \left(\td^{1/2}_{2k}\right)_k$ where $\td_{2k}^{1/2}\in H^{4k}(X, \Q)$ and developing Equation~\eqref{formJieao}, one obtain
\begin{equation}\label{formJieao-explicit}
    \overline{[F]} = \frac{[F]\cdot h^n}{q(h)^n\cdot c_X} \sum_{2k\le n} \left(\frac{-q(h)}{2r_X}\right)^k \frac{h^{n-2k}}{(n-2k)!}\ \overline{\td_{2k}^{1/2}}.
\end{equation}

When $X$ is of K$3^{[n]}$-type, we have $c_X =1$ and $r_X = \frac{n+3}{4}$. In particular, for $n=3$ the above formula reduces to
$$
\overline{[F]} = \frac{[F]\cdot h^3}{q(h)^3}\left(\frac{h^3}{3!} - \frac{q(h)}{3\cdot 24}hc_2\right),
$$
which is exactly Proposition~\ref{prop:ClassLagrangian}.\end{rema}

\section{EPW cubes}\label{sec:EPW}

Let $A\subset \bw3V_6$ be  a Lagrangian subspace with no decomposable vectors and such that the degeneracy locus $\sZ_A^{\ge 4}$ is empty. The double cover
$g_A\colon \widetilde \sZ_A \to \sZ_A$ defined in Equation~\eqref{eqn:EPWcube} is then   
branched over the smooth threefold $\sZ_A^{\ge 3}\subset  \sZ_A$ and the EPW cube $\widetilde \sZ_A $ is a smooth \hk\ manifold  
of K3$^{[3]}$-type. The ample class $h\coloneqq g_A^*\cO_{\Gr(3,V_6)}(1)$  satisfies $q(h)=4$. 

The smooth threefold $\sW_A\coloneqq g_A^{-1}(\sZ_A^{\ge 3})\subset \widetilde \sZ_A$ is the fixed locus of the involution $\iota_A$ associated to $g_A$. In this section, using results and numerical relations proved in Section~\ref{sec:hkK3}, we compute the Euler characteristic of $\sW_A$ as well as the cohomology class $[\sW_A]\in H^6(\sZ_A, \Q)$.
Additionally, we determine the Chern classes of $\sW_A$.

\subsection{Euler characteristic of  $\sW_A$}
The double cover $g_A$ gives the   equality 
$$\chi_{\textnormal{top}}(\widetilde\sZ_A  )-\chi_{\textnormal{top}}( \sW_A )=2\bigl( \chi_{\textnormal{top}}( \sZ_A )-\chi_{\textnormal{top}}( \sZ^{\ge 3}_A )\bigr),$$
between topological Euler characteristics. Since $\sW_A\simeq \sZ^{\ge 3}_A$, we obtain
\begin{equation}\label{eqn:relChiTopWa}
  \chi_{\textnormal{top}}(\sW_A) =2 \chi_{\textnormal{top}}( \sZ_A  )-\chi_{\textnormal{top}}( \widetilde\sZ_A  ).  
\end{equation}
  For a compact complex manifold, the topological Euler characteristic $\chi_{\textnormal{top}}$ is the top Chern class  so, in particular, $\chi_{\textnormal{top}}(\widetilde\sZ_A  )=3200$  by Equation~\eqref{chernnumbers}.

 The covering involution  $\iota_A$ of~$\widetilde\sZ_A$ acts on $H^\bullet(\widetilde\sZ_A,\Q)$ and
$$H^\bullet (\sZ_A  ,\Q)=H^\bullet (\widetilde\sZ_A  ,\Q)^+,$$
where $+$ means  the $\iota_A$-invariant part. We have
$$H^2 (\widetilde\sZ_A  ,\Q)^+=\Q h\quad ,\quad H^2 (\widetilde\sZ_A  ,\Q)^-=  h^\perp.$$
This is because $h$ is   $\iota_A$-invariant and the symplectic form of $\widetilde\sZ_A$ (whose class describes, by deformation of $(\widetilde\sZ_A,h)$, a dense open subset of the intersection of $h^\perp$ and a smooth quadric) is $\iota_A$-antiinvariant.

The discussion in Section~\ref{sec:AutOnLLV} implies that   either the action of $\iota_A$ on $H^\bullet(X, \Q)=V_{(3)}\oplus V_{(1,1)}$ is the action induced by the automorphism   $\iota_A^*\oplus \id$ of $V = H^2(X, \Q)\oplus U$ (as above, $U$ is the hyperbolic plane); or these actions coincide on the Verbitsky component $V_{(3)}$  and   are opposite on $V_{(1,1)}$.  

These two possibilities for the action of $\iota_A$ in cohomology give two possibilities for the cohomology groups of $\sZ_A$:
\begin{itemize}
    \item In the first case (natural action on $V_{(1,1)}$), we can compute $$
b_0(\sZ_A ) = b_2(\sZ_A ) = 1, \quad b_4(\sZ_A ) = 255, \quad   b_6(\sZ_A ) = 486,$$
and 
$$\chi_{\textnormal{top}}(\sZ_A) = 2(1+1+255)+486= 1000.$$
\item In the second case (\emph{opposite} action on $V_{(1,1)}$), we can compute 
$$
b_0(\sZ_A ) = b_2(\sZ_A ) = 1, \quad b_4(\sZ_A ) = 276, \quad   b_6(\sZ_A ) = 276,
$$
and 
$$\chi_{\textnormal{top}}(\sZ_A) = 2(1+1+276)+276= 832.$$
\end{itemize}

For example, we show how to determine $h^6(\sZ_A, \Q)$. We rewrite the decompositions in Section~\ref{sec:hodgepol} as  
$$
H^6(X,\Q) \cap V_{(3)}= \Sym^3 \!\bar T\oplus (\Sym^2\!\bar T\cdot h) \oplus   (\bar T\cdot h^2) \oplus \Q h^{ 3}
$$
 and 
$$
H^6(X,\Q) \cap V_{(1,1)} = \bw2{\bar T} \oplus ( h\wedge \bar T) \oplus \bw2U,
$$
where $\bar T \coloneqq h^\perp$. 

The fixed subspace for the action of $\iota_A^*$ on $H^6(X,\Q) \cap V_{(3)}$ is equal to $ (\Sym^2\!\bar T\cdot h)\oplus \Q h^3$, since $\iota^*_A$ acts on $\bar T$ as $-\id$ and fixes $h$. Since $\bar T$ has dimension $22$, this fixed subspace has dimension 254.  

We consider now the action of $\iota^*_A$ on $H^6(X,\Q) \cap V_{(1,1)}$. By Lemma~\ref{lemm:NaturalOpppositeAction}, this action is equal to $\pm \gamma_{(1,1)}$, where $\gamma_{(1,1)}$ is the action induced by the involution $\iota_A^*\oplus \id$ on $V$.
The fixed subspace for $\gamma_{(1,1)}$ is $\bw2{\bar T} \oplus \bw2U$, of dimension 232, and the fixed space for $-\gamma_{(1,1)}$ is equal to $h\wedge \bar T$, of dimension 22.

Hence, $b_6(\sZ_A)$ is equal to $254+232=486$ in the \emph{natural} case and to $254+22=276$ in the \emph{opposite} case.
 Therefore, the relation~\eqref{eqn:relChiTopWa} gives 
 \begin{equation}\label{eqn:chiWA}
     \chi_{\textnormal{top}}(\sW_A) = \begin{cases}
         -1200 & \quad \text{(natural case),}\\
         -1536 & \quad \text{(opposite case)}.
     \end{cases}
 \end{equation}

\subsubsection{The class $[\sW_A]$ in $H^6(\widetilde Z_A, \Z)$}
The class $[\sW_A]$ is a 3-dimensional Hodge class of $\widetilde \sZ_A$  hence, by Equation~\eqref{eqn:hodge}, it can be written as
$$
[\sW_A] = ah^3+bhc_2+c\eta,
$$
where $a,b,c\in \Q$, and $\eta $ is a    class of square 4 (see Section~\ref{sec:hodgepol}).

\begin{lemm}\label{lemm:ClassOfWA}
    The subvariety $\sW_A\subset \widetilde \sZ_A$ has class
    $$
        [\sW_A] =  \frac{5}8\left(3h^3-hc_2 \right)\in H^6(\widetilde \sZ_A,\Z).
    $$
 Moreover, the Euler characteristic of $W_A$ is $\chi_{\textnormal{top}}(\sW_A) = -1200$ and the {action of $\iota_A^*$ in cohomology is natural.}
\end{lemm}

\begin{proof}
Since $\sZ^{\ge 2}_A\cdot h^3 = 720$ (\cite[Section~2.3]{IKKR19}), it follows from Proposition~\ref{prop:ClassLagrangian} that the class  of $\sW_A$ in $ H^6(\widetilde \sZ^{\ge 2}_A, \Q)$ is 
$
    [\sW_A] = \frac{15}{8}h^3-\frac{5}8hc_2 + c\eta.
$
Since  $\sW_A$ is Lagrangian, it satisfies
$$
-\chi_\textnormal{top}(\sW_A)=[\sW_A]^2 = 1200 +4c^2.
$$
We compare this with the two possible values of $\chi_\textnormal{top}(\sW_A)$ computed in Equation~\eqref{eqn:chiWA} to obtain that $4c^2$ is either $0$ or $336$, hence $c=0$ and $\chi_\textnormal{top}(\sW_A) = -1200$.
\end{proof}

\begin{rema}\label{rem:atomic}
The restriction map
$$H^2 (\widetilde\sZ_A  ,\Q)\lra H^2 (\sW_A ,\Q)$$
has rank 1 and kernel $h^\perp$ {(it maps the symplectic form to $0$ and since the class of the latter describes a dense open subset of the intersection of $h^\perp$ and a smooth quadric, it vanishes on $h^\perp$; it is nonzero since the restriction of an ample class (such as $h$) is ample)}. 
{This implies that $[\sW_A]$ is an atomic Lagrangian \cite[Theorem~1.8]{Bec25}. Notice that its class belongs to the Verbitsky component, according to what is expected in \cite[Proposition~3.5]{Bec25}.}
\end{rema}

\subsection{Chern classes}\label{sec:ChernNumbers}
By the discussion after \cite[Theorem~5.6]{DK20}, and \cite[Lemma~4.3]{DK20} (applied with $S=\Gr(3,V_6)$, $k=3$, $c_1(\cL)=c_1(\cA_1)=0$, and $c_1(\cA_2)=\cO_{\Gr(3,V_6)}(-4)$, and giving an equality of Cartier divisor classes), we have 
 $\omega_{\sZ_A^{\ge 3}}=\cO_{\sZ_A^{\ge 3}}(2)$, so that   $\sZ_A^{\ge 3}$ (hence also~$\sW_A$) is a threefold  of general type.
 
   One can   use the exact sequence
    $$0\to T_{\sW_A} \to T_{\widetilde \sZ_A}\vert_{\sW_A}\to N_{\sW_A/\widetilde \sZ_A}\to 0
    $$
    and the isomorphism $T_{\sW_A} \isom N^\vee_{\sW_A/\widetilde \sZ_A}$ coming from the fact that $\sW_A\subset \widetilde \sZ_A$ is Lagrangian  to deduce the relation 
    $$
    c(T_{\sW_A})c(T^\vee_{\sW_A})=c(T_{\widetilde \sZ_A})\vert_{\sW_A} 
    $$
    between Chern classes.
    Since the odd Chern classes of $\widetilde \sZ_A$ vanish, this is equivalent to
    $$(-c_1^2+2c_2)( T_{\sW_A} )= c_2(\widetilde \sZ_A)\vert_{\sW_A}
    .$$

   Thus the Chern classes of $T_{\sW_A}$ are
    \begin{equation}\label{eqn:chernW_A}
        c_1(T_{\sW_A}) = -2h|_{\sW_A}, \quad
        c_2(T_{\sW_A}) = \frac{1}2 c_2|_{\sW_A} + 2h^2|_{\sW_A}, \quad c_3(T_{\sW_A}) = \chi_\textnormal{top}(\sW_A) = -1200.
    \end{equation}
We can then obtain the Euler characteristics of $\cO_{\sW_A}$ {and $\Omega_{\sW_A}^1$}. 
    
 \begin{prop}\label{prop:ChiO_W_A}
    One has $\chi(\sW_A, \cO_{\sW_A})=-130 $ {and $\chi(\sW_A, \Omega_{\sW_A}^1)   = 470$}. 
 \end{prop}
 
 \begin{proof}
 The  Hirzebruch--Riemann--Roch formula gives 
     $$
        \chi(\sW_A, \cO_{\sW_A})=\frac{1}{24}\,c_1(T_{\sW_A})c_2(T_{\sW_A}) = -\frac{1}{24}\,(4h^3+hc_2)\cdot [\sW_A]=-130.
    $$
   Since $\sW_A$ is a smooth projective threefold, by Poincar\'e duality and Hodge theory, we can write
\begin{align*}
            \frac{1}2\chi_\textnormal{top}(\sW_A) &= 1 -2h^{0,1}(\sW_A)+(2h^{0,2}(\sW_A)+h^{1,1}(\sW_A))  
         - (h^{0,3}(\sW_A)+h^{1,2}(\sW_A))\\
         &=\chi(\sW_A, \cO_{\sW_A}) -h^{1,0}(\sW_A)+h^{1,3}(\sW_A)+h^{1,1}(\sW_A)-h^{1,2}(\sW_A)\\
         &=\chi(\sW_A, \cO_{\sW_A})-\chi(\sW_A, \Omega_{\sW_A}^1),
\end{align*}
 from which we obtain, using Lemma~\ref{lemm:ClassOfWA}, the value $\chi(\sW_A, \Omega_{\sW_A}^1)   = -130-\frac12\,(-1200)=470$.
  \end{proof}

  \section{A singular degeneration of $\widetilde \sZ^{\ge 2}_A$}\label{sec:deg}

    In this section, we construct  singular degenerations of $\widetilde \sZ_A$. Our aim is to study the fixed locus of $\iota_A$ following the degeneration argument used in \cite{FMOS22} and \cite{FMOS23}. One family of singular EPW cubes was studied in \cite{Riz25}: their periods dominate the Heegner divisor ${}^{[3]}\cD_{4, 12}^{(2)}$ in the period domain. For the family that we construct below, the periods dominate another Heegner divisor, ${}^{[3]}\cD_{4, 2}^{(2)}$, in the same period domain (see \cite[Appendix~A]{Riz25}) for the notation).

Let $(S, L)$ be a {\em very general} polarized K3 surface of degree $2$. The polarization $L$ induces a double cover $\varphi\coloneqq\varphi_L \colon S \to \P^2$ branched over a smooth sextic $\Gamma\subset \P^2$, and defines an involution $j$ of $S$.
We have
$$\NS(S^{[3]})
=\Z L\oplus \Z\delta.
$$
By \cite[Proposition~4.2 and Theorem~1.1]{BC22},  the closed movable and nef cones of~$S^{[3]}$ coincide and, by \cite[Proposition~13.1]{BM14a}, they are spanned by the classes $ L$ and $ 2L-\delta$. The   class $ 2L-\delta$, of square $4$ and divisibility $2$, is therefore nef and big, but nonample on $S^{[3]}$. Therefore, the extremal ray $\R_{>0}(2L-\delta)$ induces a projective contraction
\begin{equation}\label{eqn:DefContrS3}
    \pi\colon S^{[3]}\lra X.
\end{equation}
We show in Lemma~\ref{lemm:divisorialContr} that this contraction is divisorial and describe its exceptional divisor $\Delta\subset S^{[3]}$.

Let $\bar L$ be the ample line bundle on the (singular) variety $X$ such that $\pi^*\bar L = L$.  
  The biregular antisymplectic involution $j^{[3]}$ of~$S^{[3]}$  induced by the involution $j$ of $S$ acts trivially on $\NS(S^{[3]}) = \Z L\oplus \Z\delta$ (and by $-\Id$ on its orthogonal complement in $H^2(S^{[3]},\Z)$), so it descends to an involution $\iota$ on $X$.

  The data $S^{[3]}$, $L$, $\Delta$, and $j^{[3]}$ satisfy requirements (a)--(d) of \cite[Section~2]{FMOS22}. Hence, we can apply \cite[Proposition~2.1 and Corollary~2.2]{FMOS22} (as in  \cite[Section~3.1]{FMOS23}) to obtain a family 
  \begin{equation}\label{eqn:EPWdeg}
      f_\cX \colon \cX \lra D
  \end{equation} 
  over a smooth pointed curve $(D,0)$. There is a relatively ample line bundle   $\cL$ on $\cX$ and a fiberwise involution $\iota_\cX$ of $\cX$  such that:
  \begin{itemize}
      \item[--] the   fiber over $0$ is the variety $X$ with the ample line bundle $\bar L$ and involution $\iota$;
      \item[--] for any $t\in D\setminus\{0\}$, 
      $(\cX_t, \cL_t)$ is a smooth EPW cube $\widetilde \sZ_{A_t}$ (for some Lagrangian $A_t\subset \bw3V_6$) and $\iota_t$ is the natural EPW involution  $\iota_{A_t}$.
  \end{itemize}

  In this section, we study the special fiber $(X, \iota)$ and the fixed locus of $\iota$.
\subsection{The special fiber of $\cX \to D$}\label{sec:contract}
The central fiber of the family $\cX \to D$ is the singular variety $X$   obtained by the contraction \eqref{eqn:DefContrS3} of $S^{[3]}$ defined by the nef and big extremal ray $\R_{\ge0}(2L-\delta)$.

The variety $S^{[3]}$ is the moduli space $M_S(1, 0, -2)$ of stable sheaves on $S$ with Mukai vector $v\coloneqq(1, 0, -2)$. {Recall that the vector $v$ lies in the Mukai lattice $\widetilde \Lambda=H^0(S, \Z)\oplus H^2(S, \Z)\oplus H^0(S, \Z)$, and consider the orthogonal space $v^\perp\subset \widetilde \Lambda$.} There is an identification
$$
    \Theta \colon v^\perp \isomlra H^2(S^{[3]}, \Z)
$$
such that $L = \Theta(0, -L, 0)$ and $\delta =  \Theta(-1, 0, -2)$ (\cite[Section~13]{BM14a}).

In \cite{Bri08}, Bridgeland showed that $S^{[3]}$ is also isomorphic to the moduli space $M_\sigma(v)$ of $\sigma$-stable objects in $D^b(S)$, where $\sigma$ is a stability condition contained in the distinguished connected component $\Stab^\dagger(S)$ (see \cite[Section~3]{Bay18} for the definition). 

Contractions of moduli spaces of sheaves on K3 surfaces are described in  \cite{BM14a}. Their main result is that the wall and chamber decomposition of the movable cone of $M_\sigma(v)$ corresponds to the wall and chamber decomposition induced by $v$ on $\Stab^\dagger(S)$.

 Moreover, the contraction associated to a certain wall $\cW\subset \Stab^\dagger(S)$ is characterized via \cite[Theorem~5.7]{BM14a} in terms of the associated hyperbolic lattice $H_\cW\in H^*_{\alg}(S, Z)$ (see \cite[Proposition~5.1]{BM14a} for the definition of $H_\cW$).
By definition, the hyperbolic lattice $H_\cW$ contains $v$, and $\cW$ corresponds to the wall $\Theta(H_\cW^\perp\cap v^\perp)$ of the movable cone of $M_\sigma(v)$.  

In particular, in our case, contractions of $S^{[3]}$ are given by the two extremal rays $\R_{\ge0}L$ and $\R_{\ge0}(2L-\delta)$ of $\Nef(S)$.
 The extremal ray $\R_{\ge0}L$ is orthogonal to the isotropic class $w\coloneqq (0,0,-1)$, which satisfies $w\cdot v=1$. The induced contraction is the Hilbert--Chow morphism $S^{[3]}\to S^{(3)}$ (\cite[Theorem~5.7(a)]{BM14a}). 
 
The next lemma describes 
the contraction $\pi$ induced by the other extremal ray $\R_{\ge0}(2L-\delta)$.

\begin{lemm}\label{lemm:divisorialContr}
Let $(S,L)$ be a polarized K3 surface of degree $2$. The extremal ray $\R_{\ge0}(2L-\delta)$ induces a divisorial contraction $\pi \colon S^{[3]}\to X$ with smooth exceptional divisor
    $$
        \Delta = \{Z\in S^{[3]}\mid Z  \textnormal{ is contained in some curve } C\in |L| \}.
    $$
    Moreover, the restriction of $\pi$ to $\Delta$ is a $\P^1$-fibration
    $$
        \pi|_\Delta \colon \Delta \lra T,
    $$
    where the smooth \hk\ fourfold $T$ is  the moduli space of stable sheaves on $S$ with Mukai vector $(0, L, -4)$. 
\end{lemm}

\begin{proof} 
We use the notation of \cite{BM14a}. Under the isomorphism $\Theta$, the class $2L-\delta$ is the image of the vector $w = (1, -2L, 2)$. In particular, the hyperbolic lattice associated to the contraction induced by $2L-\delta$ is $H = \langle v, (1,-L,2)\rangle = w^\perp\subset H^2_{\alg}(S, \Z)$. It is a lattice with intersection form $\begin{pmatrix} 4 & 0 \\ 0 & -2\end{pmatrix}$, and contains the spherical class $s=(1,-L,2)$.

The corresponding potential wall $\cW=\cW_H$ inside the connected component $\Stab^\dagger(S) \subset \{(\alpha, \beta)\in \R_{>0}\times\R\}$ is a connected component of the (semi)-circular wall cut out by the equation 
\begin{equation}\label{eqn:wall}
(\beta+2)^2+\alpha^2=2.
\end{equation}
Note that, by \cite[Lemma~6.2]{Bri08}, every $(\alpha, \beta)\in \R_{>0}\times \R\setminus\{(1,-1\}$ that satisfies Equation~\eqref{eqn:wall} defines a stability condition, hence the wall $\cW$ is equal to
\begin{equation}\label{eqn:Defwall}
    \cW = \{\sigma_{\alpha, \beta}\mid (\beta+2)^2+\alpha^2=2, \  \alpha\in \R_{>0}\times (-\infty, -1)\}.
\end{equation}

We analyze the wall $\cW$ using \cite[Theorem~5.7]{BM14a}. 
We first remark that $H$   contains no isotropic classes: indeed a vector $u=xv+ys\in H$, with $x,y\in \Z$, has square $4x^2-2y^2$ which is never $0$. 

Moreover, the wall $\cW$ is not totally semistable. Since $\cW$ is nonisotropic, to check this, we only need to show that there are no \emph{effective} spherical classes $u$  with $(u, v)<0$. A spherical class $u\in H$ is a vector $u=xv+ys$ with $x,y\in \Z$ such that $2x^2-y^2=-1$, and the condition $(u,v)<0$ is equal to $x<0$. 
By \cite[Proposition~5.5]{BM14a}, a class $u$ is effective if and only if $\Re\left(\frac{Z_{\alpha,\beta}(u)}{Z_{\alpha,\beta}(v)}\right)>0$ for one (hence for all) $(\alpha, \beta)\in \cW$. In particular, we can choose $\beta = -2$ and $\alpha = \sqrt{2}$.

 Using \cite[Equation~(3)]{Bay18}, we compute $Z_{\alpha, \beta}(v) = -2\beta(\beta+2+i\alpha)$ and $Z_{\alpha, \beta}(s) = -2(\beta+1)(\beta+2+i\alpha)$ for $(\alpha, \beta)\in \cW$. We obtain
$$
\frac{Z_{\sqrt{2},-2}(u)}{Z_{\sqrt{2},-2}(v)} = x + y\frac{Z_{\sqrt{2},-2}(s)}{Z_{\sqrt{2},-2}(v)} = x + \frac{y}2. 
$$
Note that for each pair $(x,y)$ with $x<0$ that satisfies $2x^2-y^2 =-1$, we have  $x+\frac{y}2 < 0$, thus there are  no effective spherical classes $u$ with $(u,v)<0$. 

Therefore the wall $\cW$ is nonisotropic not totally semistable, and the lattice $H$ contains the spherical class $s$, which is effective on $\cW$ since $\frac{Z_{\sqrt{2},-2}(s)}{Z_{\sqrt{2},-2}(v)}=1/2$. By \cite[Theorem~5.7(a)]{BM14a}, the associated contraction $\pi$ is divisorial of Brill--Noether type. In particular, by \cite[Lemma~7.1]{BM14a}, $\pi$ contracts a divisor $\Delta$ of class 
$\Theta(s) = L-\delta$. Moreover, since $s$ is the Mukai vector of the line bundle $\cO_S(-L)$, the divisor
   $\Delta$ is given by the condition $\Hom(\cO_S(-L), \bullet)\neq 0$. In particular, this corresponds to length-3 subschemes $Z\subset S$ that are contained in some curve $C\in |L|$.

More precisely, following the proof of \cite[Lemma~7.4]{BM14a}, the image $T \coloneqq\pi(\Delta)$ is isomorphic to the moduli space $M_{\sigma_0}(a)$, where $a\coloneqq v-s=(0,L, -4)$ and $\sigma_0 = \sigma_{\alpha_0, \beta_0}$ is a stability condition lying on the wall $\cW$. By \cite[Example~9.7]{BM14}, we have equality  $M_{\sigma_0}(a)=M_S(a)$ (by choosing $A=1$, $B = \beta_0$), hence  $T$ is a smooth hyper-K\"ahler fourfold that parametrizes sheaves of the form $i_{C,*}(\xi)$, where $i_C\colon C\hra S$ is the inclusion of a curve $C\in |L|$ in $S$ and $\xi$ is a torsion-free rank $1$  sheaf on $C$ of degree $-3$. 

Finally, the fiber of $\pi$ over a point $[A]\in M_{S}(a)$ is isomorphic to $\P(\Ext^1(\cO_S(-L), A))$. Since 
    $$
        \mathrm{ext}^1(\cO_S(-L), A) = \langle s, v-s\rangle = -s^2 = 2,
    $$
 the map $\pi\vert_E\colon E\to M_S(a)$ is a $\P^1$-fibration.

 If $C\in |L|$ is smooth (of genus 2), a torsion-free rank $1$  sheaf on $C$ of degree $-3$ is of  the form $\cO_C(-Z)$, with $Z\in C^{(3)}$. If $A \coloneqq i_{C, *}(\cO_C(-Z))$,  the preimage $\pi^{-1}([A])$ is equal to $|\cO_C(Z)|$ (they have the same dimension by Riemann--Roch, and any degree $3$ effective divisor  on $C$ equivalent to $Z$ defines a point in $S^{[3]}$ that is sent to $[A]$ via $\pi$).
\end{proof}

 We keep the notation of the proof.
Let $\sigma_0$ be a stability condition on the wall $\cW$ defined in Equation~\eqref{eqn:Defwall}. The  moduli space $M_{\sigma_0}(v)$ has been constructed as a good moduli space in the sense of Alper (see \cite[Theorem~3.3]{Taj23}). By \cite[Proposition~8.1]{BM14}, there exists a morphism 
\begin{equation}\label{eqn:MorfToBri}
    f\colon X\lra M_{\sigma_0}(v) 
\end{equation}
which is the identity on $X\setminus \pi(\Delta)\simeq M_{\sigma}(v)\setminus \Delta$ to $[F]$, and that sends $[F]\in M_S(a)\simeq \pi(\Delta)$ to the polystable object $[\cO_S(H)\oplus F]$.

In particular, the morphism $f$ is a bijection of sets. We will show in Section~\ref{sec:localDescr} that $M_{\sigma_0}(v)$ is normal; this will imply that $f$ is an isomorphism by Zariski's Main Theorem.

\subsection{Fixed locus on the special fiber}\label{sec:fixed}
We determine the set-theoretic fixed locus of the involution $\iota$ of $X$.
We still denote by $\Gamma\subset S$ the ramification locus of $\phi$, that is, the fixed curve of $j$.

\begin{prop}\label{prop:fixedLocus}
    The fixed locus of the involution $\iota$ of $ X$ has two connected components:
    $$
            \Fix(\iota) = {F_2}\sqcup {F_3},
    $$
    where
    \begin{itemize}
        \item ${F_2}$ is irreducible of dimension 2 and contained in the fourfold $T=\pi(\Delta)$;
         \item ${F_3}$ is irreducible of dimension 3 and its normalization is given by the restriction  
        $$
           \pi\vert_{\Gamma^{(3)}} \colon \Gamma^{(3)}\lra {F_3},
        $$
         which is an étale double cover on the smooth divisor $\Delta\cap\Gamma^{(3)}\subset \Gamma^{(3)}$.
    \end{itemize}
\end{prop}

\begin{proof}
     We start by computing the fixed locus of $j^{[3]}$ on $S^{[3]}$.  Since $S^{[3]}$ is a smooth \hk \ variety and $j^{[3]}$ is an antisymplectic regular involution, $\Fix(j^{[3]})$ is smooth Lagrangian. In particular, each component has dimension $3$. 

    Clearly, $\Gamma^{[3]}\subset S^{[3]}$ is fixed by $j^{[3]}$, hence it is one connected component of $\Fix(j^{[3]})$.  Therefore we can write $$\Fix(j^{[3]}) = \Gamma^{[3]}\sqcup F'.$$ Note that $F'$ is contained in the exceptional divisor $\Delta$. Indeed, if $Z\in S^{[3]}$ is fixed by $j^{[3]}$, then its support $\Supp(Z)$ is either contained in $\Gamma$, or of the form $\Supp(Z)=\{x+j(x)+y\}$. In particular, if $Z\notin\Gamma^{[3]}$  is fixed by $j^{[3]}$, then $\varphi(Z)$ is supported on two points, hence lies on a line in $\P^2$, so that $Z\in\Delta$.

    This characterizes the fixed locus of $\iota$ outside the smooth fourfold $W=\pi(\Delta)\subset X$. We study the action of $\iota$ on $T=M_S(0,L,-4)$. 

    \begin{lemm}\label{lemm:DefOfW}
        The action of the involution $\iota$ on the smooth fourfold $T=M_S(0,L,-4)$ is given by
   \begin{equation}\label{eqn:W4DescrInv}
       \iota(i_{C, *}(\xi_{3})) = i_{C,*}(\omega_C^{\otimes -3}\otimes \xi_3^\vee  ),
   \end{equation}
   for any   $C\in |L|$  and any degree $-3$ torsion-free rank 1 sheaf $\xi_3$ on $C$.
    \end{lemm}  
    
    \begin{proof}
        Suppose that $C$ is smooth and $\xi_3 = \cO_C(-Z)$, where $Z = p+q+r\in C^{(3)}$. Then it satisfies $\iota\left(i_{C, *}\left(\cO_C(-Z)\right)\right) =i_{C, *}\left(\cO_C\left(- j^{[3]}(Z)\right)\right)$. 
        By the adjunction formula, the canonical bundle of $C$ is equal to $\cO_C(L)$. Since $\varphi\vert_C \colon C\to   \P^1$ is a double cover, for any point $p\in C$, we have isomorphisms 
        $$\omega_C\isom \cO_C(L) \isom \varphi\vert_C^{*}(\cO_{\P^1}(\varphi(p))) \isom \cO_C(p+j (p)).$$
    Therefore, we can rewrite
    $$
\cO_C\left(- j^{[3]}(Z)\right) = \left(\cO_C(-j(p)-j(q)-j(r))\right) =\omega_C^{\otimes -3}\otimes\cO_C(p+q+r).
    $$
    In particular,
      formula~\eqref{eqn:W4DescrInv} is verified for any element of $T $ of the form $i_{C, *}(\cO_C(-p-q-r))$ where $p,q,r$ are points on $C$ and $C\in |H|$ is a smooth curve. 

        By Bertini's theorem, these elements form a dense open subset of $T$, therefore the involution on $T$ is given by   formula~\eqref{eqn:W4DescrInv} everywhere.  
    \end{proof}

    Note that upon tensoring by $\cO_S( L)$, we obtain an isomorphism
    $$ 
    T \isom M_S(0,L,0),
    $$
    where elements of $M_{S}(0,L,0)$ are   isomorphism classes of sheaves $i_{C,*}(\xi)$, where $i_C\colon C\to S$ is the inclusion of a curve $C\in |L|$ and $\xi$ is a torsion-free sheaf of rank $1$ and degree $1$ on $C$.

    Under this isomorphism, the involution $\iota$ on $T$ corresponds to the involution $\iota_0$ on $M_{S}(0,L,0)$  given by
    $$
\iota_0(i_{C, *}(\xi)) = i_{C, *}( \omega_C\otimes  \xi^\vee ),
    $$
    which is one of the involutions studied in \cite{FMOS22}.
    
    Using the description of the fixed locus of $\iota_0$ in \cite[Proposition~4.2]{FMOS22}, we obtain that the fixed locus of $\iota\vert_{T}$ has two  connected components $W_+$ and $W_-$, both smooth irreducible and 2-dimensional, which are the closures of the loci of $i_{C, *}(\xi\otimes \cO(-L))$ where $\xi$ is an even, or odd, theta characteristic on a smooth curve $C$. Since $C$ has genus $2$, a theta characteristic $\xi$ on   $C$ is odd if and only if $h^0(C, \xi)=1$ (the 6 odd theta characteristics correspond to the 6 Weierstrass points $x_1,\dots,x_6$ of $C$; the 10 even theta characteristics correspond to the $x_i+x_j+x_k-g^1_2$, where $i,j,k\in\{1,\dots,6\}$ are distinct---this description counts each of them twice, using $x_1+\cdots+x_6=3g^1_2$).

    To conclude, we show that $\Fix(\iota)$ is equal to the image $\pi(\Fix(j^{[3]}))$. Clearly, we have an inclusion $\pi(\Fix(j^{[3]})) \subset \Fix(\iota)$ and they are equal on    $X\smallsetminus T$. We restrict ourselves to $T$, where $\Fix(\iota)\cap T = \Fix(\iota|_{T}) = W_+\sqcup W_-$ and 
    $$
        \pi\left(\Fix(j^{[3]})\right)\cap W  = \pi\left(\Fix(j^{[3]})\cap \Delta\right) = \pi\left(\Gamma^{(3)}\cap \Delta\right)\cup \pi(F').  
    $$
   Now, $\pi(F')$ is contained in $W_+\sqcup W_-$, which has dimension 2. This implies that $F'$ is contracted by $\pi$. {We show that $\pi(F')=W_-$. 
 By definition of $W_-$, it is enough to show that for $i_{C,*}(\xi_{3})\in \pi(F')$ general, the element $\xi = \xi_3\otimes \cO_C(L)$ is an odd theta characteristic. The general element of $\pi(F')$ is $\iota_{C,*}\cO_C(-x-j(x)-y)\in T$, where $C$ is the pullback of the line through $\varphi(x)$ and $\varphi(y)$. Therefore, we obtain $\xi = \cO_C(-x-j(x)-y)\otimes \cO_C(x+j(x)) = \cO_C(-y)$, which is an odd theta characteristic by the discussion above,  since $y$ belongs to $C\cap \Gamma$.
 Therefore, the morphism $\pi$ restricts to a $\P^1$-fibration $F'\to W_-$. 
 }
    
    The component $\Gamma^{(3)}$ intersects $\Delta$ in a divisor and is disjoint from $F'$, hence $\pi(\Gamma^{(3)}\cap \Delta)$ is equal to $W_+$. {The restriction of $\pi$ to $\Gamma^{3}\cap \Delta$ is finite, since the fiber over $i_{C,*}(\xi)$ is contained in $C^{(3)}\cap \Gamma^{(3)}\subset (C\cap \Gamma)^3$ 
    Moreover, for any element $F = i_{C,*}(-Z)\in W_+$, for a smooth curve $C$, the intersection $\pi^{-1}(F) \cap \Gamma^{(3)}= |\cO_C(-Z)|\cap \Gamma^{(3)}$ is given by $2$ points, namely $Z$ and the residual intersection $Z'\coloneqq \Gamma\cdot C - Z$.} Therefore, the induced morphism 
 \begin{equation}\label{eqn:doublecover}
          \Gamma^{(3)}\cap \Delta\xrightarrow{\ 2:1\ }  W_+  
 \end{equation}
    is a double cover, and the corresponding involution sends collinear points $x+y+z\in \Gamma^{(3)}$ on a line $\ell$ to the residual intersection $\Gamma\cap \ell-(x+y+z)$.
     A general sextic $\Gamma\subset \P^2$ has no tritangent lines, so this involution is base-point-free. The double cover \eqref{eqn:doublecover} is \'etale  and, since the surface $W_+$ is smooth, so is $\Gamma^{(3)}\cap \Delta$.

    Therefore, $\pi\left(\Fix(j^{[3]})\right)\cap T = \Fix(\iota|_{T})$. In particular,   the fixed locus of $\iota$ is the image of the fixed locus of $j^{[3]}$. 

    This ends the proof of the proposition, with $F_2\coloneqq W_-$ and $F_3\coloneqq \pi(\Gamma^{[3]})$.
\end{proof}

 {
\begin{rema}\label{rema:CentralFiberSLC}
    The scheme $F_3$ is semi-log-canonical (see \cite[Chapter 3, Definition~12]{Kol23}). Indeed, its normalization is the smooth scheme $\Gamma^{(3)}$ and the normalization morphism
    $$
        \pi\colon \Gamma^{(3)} \lra F_3,
    $$
    restricts to an étale double cover over the divisor $E\subset \Gamma^{(3)}$ of collinear points of $\Gamma^{(3)}$. 
\end{rema}
}

\subsection{Local description of $X$}\label{sec:localDescr} Following \cite{FMOS23}, we give a local description of the variety $X$ and the involution $\iota$. Moreover, we show that the schematic fixed locus $\Fix(\iota)$ is a reduced local complete intersection. 

The first step is to show that the morphism $f\colon X \to M_{\sigma_0}(v)$ constructed in Equation~\eqref{eqn:MorfToBri} is an isomorphism. We have already noted that $f$ is a bijection  hence, by Zariski's Main Theorem, it is enough to show that $M_{\sigma_0}(v)$ is normal.

By the analysis in Section~\ref{sec:contract}, singular points of $M_{\sigma_0}(v)$ are given by $S$-equivalence classes of polystable sheaves of the form $E=\cO(-L)\oplus F$, where $F$ is a stable vector sheaf in the fourfold $M_S(0, L, -4)$. 

We study the local structure of $M_{\sigma_0}(v)$ at $[E]$ using the Kuranishi map.  
Consider the map
$$
\mu_2 \colon \Ext^1(E,E)\lra \Ext^2(E,E)
$$
given by the Yoneda product. By \cite[Corollary~4.1]{AS25}, there is an isomorphism of germs
\begin{equation}\label{eqn:germsKuranishi}
    (M_{\sigma_0}(v), [E]) \simeq (\mu_2^{-1}(X)/\!\!/ G, [0]),
\end{equation}
where the group $G$ is equal to  $\Aut(\cO(-L))\times \Aut(F)$ and acts on $\Ext^1(E,E)$ by conjugation.

The situation is analogous to that of \cite[Remark~2.7 and Lemma~5.4]{FMOS23}. Using their methods, we show the following lemma.

\begin{lemm}\label{lemma:NormalModuli}
    Let $[E]$ be a singular point of $M_{\sigma_0}(v)$. There is an isomorphism of analytic germs
$$
(X, [E])\simeq (\A^4_\C\times Q, 0)
$$
where $Q = \Spec\left(\C[u_1,u_2,u_3]/u_1^2-u_2u_3\right)$. 

In particular, the good moduli space $M_{\sigma_0}(v)$ is normal.
\end{lemm}

\begin{proof}
       Up to changing the singular point $[E]$ in its $S$-equivalence class, we can suppose that $E$ is of the form $T\oplus F$, where $T$ is the stable spherical sheaf $\cO(-L)$ and  $F$ is stable with Mukai vector $a=(0,L,-4)$. 
    By Equation~\eqref{eqn:germsKuranishi}, the germ of $X$ at the singular point $E$ is equal to $\mu_2^{-1}(0)/\!\!/ G$. 

    We describe the morphism $\mu_2$ on $\Ext^1(E,E)$. Since $T$ is spherical, we have decompositions  $$\Ext^1(E,E)=\Ext^1(F,F)\oplus \Ext^1(T,F)\oplus \Ext^1(F,T),$$ 
    and
    $$
    \Ext^2(E,E)=\Ext^2(F,F)\oplus \Ext^2(T,T)\simeq \C^2,
    $$
    where we used that $\Ext^2(T,F)=\Ext^2(F,T)=0$, since $T$ and $F$ are nonisomorphic stable sheaves of same phase.
    The Yoneda product $\mu_2$ is given by
    $$(\eta, a,b) \longmapsto(\eta^2 + b\circ a, a\circ b).  $$  

    Note that, since $F$ is a stable point of the moduli space $M_{\sigma_0}(a)$, the Kuranishi map for $F$ is trivial, hence $\eta^2=0$ for any $\eta\in \Ext^1(F,F)$. 
    Additionally, the trace map defines isomorphisms $\Ext^2(F,F)\simeq \C$ and $\Ext^2(T,T)\simeq \C$ such that $a\circ b= -b\circ a $ for any $(a, b)$ in $\Ext^1(T,F)\oplus \Ext^1(F,T)$. 
    
   Therefore, the kernel of $\mu_2$ is equal to 
    $$
        \mu_2^{-1}(0)= \Ext^1(F,F)\times\{(a,b)\in \Ext^1(T,F)\oplus \Ext^1(F,T)\mid a\circ b = 0\}.
    $$
     The vector space $\Ext^1(F,F)$ has dimension $a^2+2 = 4$, and the spaces $\Ext^1(T,F)$ and $\Ext^1(F,T)$ have dimension $a\cdot(v-a) = 2$. Given a basis  $(a_1,a_2)$ of $\Ext^1(T,F)$, let $(b_1, b_2)$ be a basis of $\Ext^2(F,T)$ such that $a_1\circ b_2= a_2 \circ b_1 = 0$. We obtain an isomorphism
     $$
    \mu_2^{-1}(0)\simeq \A_\C^4\times \Spec\left(\C[a_1,a_2,b_1,b_2]/(a_1b_1+a_2b_2)\right).
     $$
     Now the action of the group $G = \Aut(T)\oplus \Aut(F)\simeq (\C^*)^2$ is given by conjugation on elements of $\Ext^1(E,E)$. Therefore, it is trivial on $\Ext^1(F,F)$ and $(g,h)\in (\C^*)^2$ acts as the multiplication by $gh^{-1}$ on $\Ext^1(T, F)$ and as the multiplication by $hg^{-1}$ on $\Ext^1(F,T)$. Therefore, as in \cite[Remark~2.7]{FMOS23}, the quotient $ \mu_2^{-1}(0)/\!\!/ G$ is equal to $\A_\C^4 \times Q$, where
     $$
        Q = \Spec\left(\C[u_1,u_2,u_3]/(u_1^2-u_2u_3)\right),
     $$
     where $u_1=a_1b_1$, $u_2 = -a_1b_2$ and $u_3=a_2b_1$.

     This shows in particular that $M_{\sigma_0}(v)$ is   a local complete intersection. Since its singular locus is parametrized by $M_S(v)$, it has codimension 2. Therefore the moduli space $M_{\sigma_0}(v)$ is normal.
\end{proof}

As already noted, the above lemma implies the wanted modular description for the special fiber of $\cX\to B$.

\begin{coro}\label{coro:XisModuli}
    The morphism $f$ given in Equation~\eqref{eqn:MorfToBri} is an isomorphism, hence the central fiber $X$ of the family $\cX\to B$ is isomorphic to $M_{\sigma_0}(v)$.
\end{coro}

This allows us to determine the local structure of the schematic fixed locus $\Fix(\iota)$.

\begin{lemm} Let $[E]$ be a singular point of $X$  and consider the   isomorphism of germs
$$
(X, [E])\simeq (\A^4_\C\times Q, 0)
$$
from Lemma~\ref{lemma:NormalModuli}.
  If the point $[E]$  lies on the component $F_3$ or $\Fix(\iota)$ (see Proposition~\ref{prop:fixedLocus}), this isomorphism restricts to an isomorphism of germs
$$
(\Fix(\iota), [E])\simeq (\A^2_\C\times B, 0),
$$
where $B =  \Spec\left(\C[u_2,u_3]/(u_2u_3)\right)$. In particular, $\Fix(\iota)$ is reduced at points of $F_3$.
\end{lemm}

\begin{proof}
    By Corollary~\ref{coro:XisModuli}, the variety $X$ is isomorphic to $M_{\sigma_0}(v)$, and this gives the first isomorphism of germs. We go back to the notation of the proof of Lemma~\ref{lemma:NormalModuli}.
    
     We start by describing the action of the involution on the germ $\A_\C^4 \times Q$. Since $[E]$ is fixed by $\iota$, the factor $F\in M_S(a)$ is fixed by the restriction of $\iota$ to $M_S(a)=W\subset X$ (see Section~\ref{sec:contract}).
     Since $M_S(a)$ is a smooth \hk \ fourfold and $\iota$ restricts to an antisymplectic involution on it, its fixed locus is smooth of dimension 2. Hence, on germs, the fixed locus of $\iota|_{\A^4_\C}$ is equal to $\A_\C^2$.
     
     We now analyze the action of $\iota$ on $Q$.  
     Note that the fixed eigenspace for the action of $\iota$ on $\Ext^1(T,F)$ has dimension $1$.
     Indeed, since $\P(\Ext^1(T,F))$ coincide with the fiber of $\pi$ over the point $E$ (see Lemma~\ref{lemm:DefOfW}), this is equivalent to say the fixed locus of $j^{(3)}|_{\P(\Ext^1(T,F))}$ is finite (when $F$ is $j^{(3)}$-invariant), which follows from the proof of Proposition~\ref{prop:fixedLocus}.
    
    Moreover the action of $\iota$ on $\Ext^1(F,T)$ is dual to the one on $\Ext^1(T,F)$, therefore the action of $\iota$ on $Q$ is the one described in \cite[Proposition~5.4]{FMOS23}. 
     In particular, its fixed locus is equal to $B = \Spec(\C[u_2,u_3]/u_2u_3)$. 

     This implies that the schematic fixed locus $\Fix(\iota)$ is reduced. Indeed, it is smooth at any smooth point of $X$, since the action $\iota$ restricts to a regular action on the smooth locus of $X$, and the reducedness at singular points follows from its local description, since $B$ is reduced.
\end{proof}

\begin{rema}
    A similar argument shows that the whole fixed locus of $\iota$ is reduced. 
\end{rema}

\section{Degeneration of $\sW_A$}\label{sec:DegWA}

Given the family $\cX \to D$ studied in Section~\ref{sec:deg},
   we consider the schematic fixed locus $\Fix(\iota_\cX)\to \cD$ of $\iota_\cX$ on $\cX$ (see the proof of \cite[Proposition~5.6]{FMOS23} for the definition of the schematic structure on the fixed locus). For $t\neq 0$, the fiber $\Fix(\iota_\cX)_t = \Fix(\iota_{t})$ is the smooth Lagrangian threefold $\sW_{A_t}$. The reducedness of the fixed locus of $\iota $ studied in Section~\ref{sec:localDescr} implies $\Fix(\iota_\cX)_0 = \Fix(\iota) = F_3\cup F_2$.
   
   Let $\cX^*$ be the pullback of  $\cX\to D$ to $D\setminus\{0\}$. The family $\cX^*$ is smooth, hence the fixed locus $\Fix_{\cX^*}(\iota_{\cX^*})$ of $\iota_{\cX^*}$ is also smooth and its fibers are equal to $\Fix(\iota_{t})$ for $t\neq 0$.
  Let $\cY$ be the closure of $\Fix_{\cX^*}(\iota_{\cX^*})$ inside $\cX$, with induced morphism 
  \begin{equation}\label{eqn:DefOfW_A}
      f \colon \cY\lra D.
  \end{equation}
  The morphism $f$ is flat by \cite[4.3~Proposition~3.9]{Liu02} since it is dominant and $\cY$ is integral. Therefore, it is equidimensional and $\cY_0$ has dimension $3$. Moreover, $\cY$ is contained in $\Fix(\iota_\cX)$, hence $\cY_0$ is equal to $F_3$ (scheme-theoretically).

\begin{theo}\label{thm:HodgeNumbers}
    Let $A\subset \bw3V_6$ be  a Lagrangian subspace with no decomposable vectors, and such that the degeneracy locus $\sZ_A^{\ge 4}$ is empty.
    The 
    external Hodge numbers of $\sW_A$ can be computed as
     \begin{equation*}
        h^{p,0}(\sW_A)=h^p(F_3, \cO_{F_3}).
    \end{equation*}
\end{theo}

\begin{proof}

    Hodge numbers are invariant under smooth deformations, hence it is enough to compute $h^{p,0}(\sW_A)$ when $\sW_A$ is a smooth fiber of the morphism $f\colon \cY\to D$ defined in \eqref{eqn:DefOfW_A}.

   All fibers of $f$ are slc: the central fiber is isomorphic to $F_3$, which is slc by Remark~\ref{rema:CentralFiberSLC}, and the other fibers $\sW_A$ are smooth. Therefore, since $f$ is flat and proper, \cite[Corollary~2.64]{Kol23} implies that $h^p(\cY_t, \cO_{\cY_t})$ is independent of $t\in D$.
   The theorem  follows.
\end{proof} 

As a first consequence, the above theorem allows us to prove the following result (already stated in \cite{Riz25,DK20} without proof).

\begin{coro}\label{Z3irred}
    Let $A\subset \bw3V_6$ be a Lagrangian subspace containing no decomposable vectors. The schemes $\sW_A$ and $\sZ^{\ge 3}_{A}$ are irreducible.
\end{coro}

\begin{proof}
\cite[Corollary 2.10]{IKKR19} shows that the scheme $\sZ^{\ge 3}_{A}$ is normal and reduced, therefore it is irreducible if and only if it is connected (see \cite[Theorem~2.5]{Riz26a} for more details).
Since $\sW_A$ and $\sZ^{\ge 3}_A$ are isomorphic, it suffices to prove that $\sW_A$ is connected.  If $A$ contains no decomposable vectors and the degeneracy locus $\sZ_A^{\ge 4}$ is empty,   the connectedness of $\sW_A$ follows from the connectedness of $F_3$ via Theorem~\ref{thm:HodgeNumbers}. 

We use a ``standard trick,'' as in \cite[Theorem~B.2]{DK18}, to show that $\sZ^{\ge 3}_A$ is connected for any $A $ containing no decomposable vectors and for which $\sZ^{\ge 4}_A\neq \varnothing$. Let $(T,0)$ be a smooth pointed connected curve parametrizing Lagrangians $(A_t)_{t \in T}$ with no decomposable vectors, such that $\sZ^{\ge 4}_{A_t}$ is empty for $t \neq 0$, and $A_0 = A$. For any $t \neq 0$, the scheme $\sZ^{\ge 3}_{A_t}$ is a smooth irreducible threefold. Consider the subvariety
$$
\cZ^{\ge 3}_T \coloneqq \bigcup_{t \in T} \ \{t\} \times \sZ^{\ge 3}_{A_t}
$$
inside $T \times \Gr(3,V_6)$. It is a third determinantal locus, therefore it has codimension at most $6$ everywhere. Hence each connected component of $\cZ^{\ge 3}_T$ has dimension at least~$4$. Since the scheme $\sZ^{\ge 3}_{A_0}$ has dimension $3$, it follows that every connected component of $\sZ^{\ge 3}_{A_0}$ deforms in a neighborhood of $0$. As we have already proved that~$\sZ^{\ge 3}_{A_t}$ is connected for $t \neq 0$, we conclude that $\sZ^{\ge 3}_{A_0}$ is connected as well.
\end{proof}

\begin{rema}
    Corollary~\ref{Z3irred} allows us to generalize Theorem~\ref{thm:HodgeNumbers} to Lagrangian subvarieties $A$ containing no decomposable vectors with $\sZ^{\ge 4}_{A}\neq \varnothing$. As in the proof of  Corollary~\ref{Z3irred}, consider the scheme $\cZ^{\ge 3}_T$, where $(T,0)$ is a smooth pointed connected curve parametrizing Lagrangians $(A_t)_{t \in T}$ with no decomposable vectors, such that $\sZ^{\ge 4}_{A_t}$ is empty for $t \neq 0$, and $A_0 = A$.
    
    Again by \cite[4.3~Proposition~3.9]{Liu02}, the morphism
    $$
        \pr_T \colon \cZ^{\ge 3}_T \lra T
    $$  
    is flat and proper, since $T$ is a curve and the proof of Corollary~\ref{Z3irred} shows that $\cZ^{\ge 3}_T$ is integral.
 
    The fiber of $\pr_T$ at $t\in T$ is the scheme $\sZ^{\ge 3}_{A_t}$, which is smooth for every $t\neq 0$. By \cite[Theorem~2.1 and Proposition~2.6]{Riz25}, the singular locus of $\sZ^{\ge 3}_{A_0}$ is the finite set $\sZ^{\ge 4}_{A_0}$, and the singularity of $\sZ^{\ge 3}_{A_0}$ at each singular point is the singularity of a cone over a general hyperplane section of the Veronese $\nu_2(\P^3)$. In particular, the scheme $\sZ^{\ge 3}_{A_0}$ is normal and has canonical singularities \cite[Examples~7.16(1)]{Deb01}. Thus all fibers of $\pr_T$ are slc. Applying again \cite[Corollary~2.64]{Kol23}, we conclude that $h^p(\sZ^{\ge 3}_{A_t},\cO_{\sZ^{\ge 3}_{A_t}})$ is independent of $t$. 
    This proves that the equality $h^p(\sW_A, \cO_{\sW_A}) = h^p(F_3, \cO_{F_3})$ holds for every Lagrangian $A\subset \bw3V_6$ containing no decomposable vectors.
\end{rema}

\subsection{Cohomology groups $H^p(F_3, \cO_{F_3})$}
We study the cohomology of the structure sheaf of~$F_3$. We keep the notation of Section~\ref{sec:fixed}.

Let $\Gamma\subset \P^2$ be a smooth plane sextic and let $E\subset \Gamma^{(3)}$ be the (smooth) divisor of collinear 3-points on $\Gamma$, with involution $\sigma$ defined by 
$$
\sigma(x+y+z)= (C\cdot \ell) - x-y-z,
$$
for any $x+y+z\in E$ that lies on a line $\ell$.

By Proposition~\ref{prop:fixedLocus}, the scheme $F_3$ is obtained from $\Gamma^{(3)}$ via the morphism
\begin{equation}\label{eqn:MorfF_3}
    \eta \colon \Gamma^{(3)} \lra F_3
\end{equation}
given by the restriction of the contraction $\pi \colon S^{[3]}\to X$ to $\Gamma^{(3)}\subset S^{[3]}$. The morphism $\eta$ is an isomorphism on $\Gamma^{(3)}\setminus E$, and $\eta|_E \colon E \to W\coloneqq \eta(E)$ is the (étale) double cover induced by the involution $\sigma$ defined above. Notice that $\eta$ is in particular the normalization of $F_3$.

Consider the morphism of sheaf $\cO_{F_3}\to \eta_*\cO_{\Gamma^{(3)}}$. It is injective, as $\eta$ is dominant. Moreover, since $\eta$ is an isomorphism outside $E$, the cokernel of $\eta$ is supported on $W$. Since the restriction $\eta|_{E}$ is a double cover, there is a decomposition
$$
(\eta\vert_{E})_*\cO_E = \cO_W \oplus \cR,
$$
where $\cO_W$ and $\cR$ are respectively the invariant and anti-invariant parts with respect to the action of $\sigma^*$.
Since $\cO_{F_3}\vert_W = \cO_W$ and $\eta_*\cO_{\Gamma^{(3)}}\vert_W=(\eta\vert_{E})_*\cO_E$, we have a short exact sequence
\begin{equation}\label{eqn:sesF3}
    0 \lra \cO_{F_3} \lra \eta_*\cO_{\Gamma^{(3)}}\lra j_*\cR\lra 0
\end{equation}
where $j$ is the embedding $W\hra F_3$  and the map $\eta_*\cO_{\Gamma^{(3)}}\to j_*\cR$ is induced by taking the restriction to $W$ and projecting   to $\cR$. This induces a long exact sequence in cohomology
\begin{align}\label{eqn:longexactseq}
 0&\to H^{1}(F_3, \cO_{F_3})\to  H^{1}(\Gamma^{(3)}, \cO_{\Gamma^{(3)}})\xrightarrow{\ \beta_1\ }  H^{1}(W, \cR)\to H^2(F_3, \cO_{F_3}) \to    \\
&\to H^2(\Gamma^{(3)}, \cO_{\Gamma^{(3)}})\xrightarrow{\ \beta_2\ } H^2(W, \cR) \to  H^{3}(F_3, \cO_{F_3})\to H^{3}(\Gamma^{(3)}, \cO_{\Gamma^{(3)}}) \to  0,\nonumber
    \end{align}
 where we used that $H^3(W, \cR)=0$ since $W$ has dimension $2$, and $H^0(W, \cR) = 0$ because of the equality $\C=H^0(E, \cO_E)=H^0(W, (\eta\vert_{E})_*\cO_E)=H^0(W, \cO_W)$.

The map $\beta_i\colon H^{i}(\Gamma^{(3)}, \cO_{\Gamma^{(3)}}) \to H^{i}(W, \cR)$ factors as
$$
H^{i}(\Gamma^{(3)}, \cO_{\Gamma^{(3)}}) \xrightarrow{\ \res_i\ }   H^{i}(E, \cO_E)=H^i(W, (\eta\vert_{E})_*\cO_E) \lra H^{i}(W, \cR),
$$
where the morphism $\res_i$ is induced by the restriction to $E\subset \Gamma^{(3)}$ and the second map is induced by the projection $(\eta\vert_{E})_*\cO_E\to \cR$.

\begin{lemm}\label{lemm:resInj}
    The restrictions 
    $$
        \res_{0,q} \colon H^0(\Gamma^{(3)}, \Omega_{\Gamma^{(3)}}^{ q}) \lra H^0(E, \Omega_{E}^{ q})
    $$
    are injective for $q \in\{0, 1,2\}$. 
\end{lemm}

\begin{proof}
     Since $\Gamma$ is a smooth plane sextic curve, it is not trigonal, hence the symmetric product $\Gamma^{(3)}$ embeds into $\Jac(\Gamma)$ and the restrictions
\begin{equation}\label{eqn:restriction}
    H^0(\Jac(\Gamma),\Omega^q_{\Jac(\Gamma)})\lra H^0(\Gamma^{(3)} ,\Omega^q_{\Gamma^{(3)} })
\end{equation}
are isomorphisms for $q\in\{0,1,2,3\}$
 (\cite[(11.1)]{Mac62}). Therefore, it is enough to consider the restriction morphisms induced by the inclusion $E\subset \Gamma^{(3)}\subset \Jac(\Gamma)$.
 Moreover, since
 \begin{equation}\label{eqn:wedge}
 H^0(\Jac(\Gamma),\Omega^q_{\Jac(\Gamma)})=\bw{q} H^0(\Jac(\Gamma),\Omega^1_{\Jac(\Gamma)}),
 \end{equation}
 it is enough to show the injectivity of the restriction
   $$r\colon H^0(\Jac(\Gamma),\Omega^2_{\Jac(\Gamma)})\lra H^0(E,\Omega^2_E).$$ 
    We show that the surface $E\subset \Jac(\Gamma)$ is nondegenerate in the sense of \cite[Section~II]{Ran80}, which implies by \cite[Lemma II.1]{Ran80} that $r$ is injective.

     Let $\theta$ be the canonical principal polarization on $\Jac(\Gamma)$.
      By \cite[(16.2)]{Mac62}, the cohomology class of $E$ in $H^2(\Gamma^{(3)},\Z)$ is equal to 
      $$
      -6\eta +\theta|_{\Gamma^{(3)}}
      \in H^2(\Gamma^{(3)},\Z),$$
      where $\eta\in H^2(\Gamma^{(3)},\Z)$ is the class of $\Gamma^{(2)}\subset \Gamma^{(3)}$.\footnote{The divisor $E\subset \Gamma^{(3)}$ is not ample. Indeed by \cite[Lemma~(1)]{Kou93}, we have  $(\theta-6\eta)^3 = \sum_{i=0}^3 {3 \choose i}(-6)^{3-i}\frac{10!}{(10-i)!} = -36$}
    By the projection formula, the class of $E$ in $\Jac(\Gamma)$ is therefore given by
    $$
    [E] = - 6[\Gamma^{(2)}]+[\Gamma^{(3)}]\cdot \theta.
    $$
    Since $[\Gamma^{(i)}]\in H^*(\Jac(\Gamma), \Z)$ is equal to $\frac{\theta^{10-i}}{(10-i)!}$, it follows that $[E]=2\theta^8/8!$ in $H^*(\Jac(\Gamma), \Z)$. Therefore, $[E]$ is a positive multiple of $\theta^8$, hence $E$ is nondegenerate by \cite[Corollary~II.2]{Ran80}.   
\end{proof}

\begin{prop}\label{prop:HdgF3}
    The cohomology groups of the structure sheaf of $F_3$ are
    \begin{align*}
        H^1(F_3, \cO_{F_3}) &= 0,\\
        H^2(F_3, \cO_{F_3}) &= H^{2,0}(\Gamma^{(3)})\oplus \coker(\res_1), \\
        H^3(F_3, \cO_{F_3}) &= H^{3,0}(\Gamma^{(3)}) \oplus H^2(W,\cR).
    \end{align*}
\end{prop} 

\begin{proof}
  We start by showing that in the long exact sequence~\eqref{eqn:longexactseq}, the morphism $\beta_1$ is injective and   $\beta_2$ is the zero map. 
    We have a factorization
    $$
    \beta_i \colon H^i(\Gamma^{(3)}, \cO_{\Gamma^{(3)}})\hra H^i(E, \cO_E) \thra H^i(W, \cR), 
    $$
    where the first map is injective  by Lemma~\ref{lemm:resInj} and Hodge theory, since $\Gamma^{(3)}$ and $E$ are smooth by Proposition~\ref{prop:fixedLocus}, and the second map is the projection to the anti-invariant part of $H^i(E, \cO_E)$ with respect to the action of $\sigma^*$.

      Consider the embedding $E\subset \Gamma^{(3)}\subset \Jac^3(\Gamma)$, as in the proof of Lemma~\ref{lemm:resInj}. By~\eqref{eqn:restriction} and~\eqref{eqn:wedge}, the involution $\sigma^*$ acts by multiplication by $(-1)^q$ on the subspace $H^0(\Gamma^{(3)},\Omega^q_{\Gamma^{(3)}})\subset H^0(E,\Omega^q_{E})$, hence also on $H^q(\Gamma^{(3)}, \cO_{\Gamma^{(3)}})\subset H^q(E, \cO_E)$. 
    
    In particular, $H^1(\Gamma^{(3)}, \cO_{\Gamma^{(3)}})$ is anti-invariant for the action of $\sigma^*$ on $H^1(E, \cO_E)$, hence is contained in $H^1(W, \cR)$. Therefore, the morphism $\beta_1$ is injective. 
    Analogously, $H^2(\Gamma^{(3)}, \cO_{\Gamma^{(3)}})$ is all invariant for the action of $\sigma^*$, therefore the morphism $\beta_2$ is the zero map.

    Therefore, the long exact sequence~\eqref{eqn:longexactseq} gives $H^1(F_3, \cO_{F_3})=0$ and exact sequences
    $$
0\to H^{1}(\Gamma^{(3)}, \cO_{\Gamma^{(3)}}) \xrightarrow{\ \beta_1\ } H^{1}(W, \cR)\to H^2(F_3, \cO_{F_3})\to H^2(\Gamma^{(3)}, \cO_{\Gamma^{(3)}})\to 0
    $$
    and
    $$
0\to H^{2}(W, \cR)\to H^3(F_3, \cO_{F_3})\to H^3(\Gamma^{(3)}, \cO_{\Gamma^{(3)}})\to 0,
    $$
    from which we deduce the wanted statements.
\end{proof}

{Using Theorem~\ref{thm:HodgeNumbers}, Proposition~\ref{prop:ChiO_W_A},~\eqref{eqn:restriction}, and~\eqref{eqn:wedge}, we obtain information on some Hodge numbers of the smooth threefold $\sW_A$.

\begin{coro}
Assume that the EPW cube $\widetilde \sZ_A$ is smooth. We have 
\begin{align*}
        h^{0,1}(\sW_A) &= 0,\\
        h^{0,2}(\sW_A) &\ge 
        45, \\
        h^{0,3}(\sW_A) &=131+ h^{0,2}(\sW_A), \\
        h^{1,2}(\sW_A) &=470+ h^{0,2}(\sW_A)+ h^{1,1}(\sW_A).
    \end{align*}
    \end{coro}
 
}

\begin{coro}
Assume that the EPW cube $\widetilde \sZ_A$ is smooth. 
Its submanifold $\sW_A$ is rigid in $\widetilde \sZ_A$. 
\end{coro}

\begin{proof}
    Since $\sW_A$ is Lagrangian, the normal sheaf $N_{\sW_A/\widetilde \sZ_A}$ is isomorphic to $\Omega_{\sW_A}^1$. Therefore, $$h^{0}(\sW_A, N_{\sW_A/\widetilde \sZ_A})=h^0(\sW_A, \Omega_{\sW_A}^1)=h^{0,1}(\sW_A)=0$$ and the submanifold $\sW_A$ does not deform in $\widetilde \sZ_A$.
\end{proof}

\bibliographystyle{amsalpha}
\bibliography{biblio}

\end{document}